\documentclass[12pt]{article}
\usepackage{amscd,amssymb,amsmath,verbatim,amsthm, graphicx, tikz-cd}

\usepackage{color}

\newcommand{\XX}{\mathbf{a}} %place holder for name

\newcommand{\bT}{\mathbb T}
\newcommand{\bS}{\mathbb S}

\newcommand{\CC}{\mathbb C}

\newcommand{\NN}{\mathbb N}
\newcommand{\RR}{\mathbb R}

\newcommand{\del}{\partial}
\newcommand{\delbar}{\overline{\del}}

\newcommand{\phg}{{\mathrm{phg}}}

\newcommand{\Diff}{\mathrm{Diff}}

\newcommand{\ran}{\mbox{ran\,}}
\newcommand{\calA}{{\mathcal A}}

\newcommand{\calC}{{\mathcal C}}

\newcommand{\calE}{{\mathcal E}}
\newcommand{\calF}{{\mathcal F}}

\newcommand{\calH}{\mathcal H}

\newcommand{\calJ}{{\mathcal J}}
\newcommand{\calK}{{\mathcal K}}
\newcommand{\calL}{{\mathcal L}}

\newcommand{\calO}{{\mathcal O}}
\newcommand{\calP}{{\mathcal P}}
\newcommand{\calS}{{\mathcal S}}

\newcommand{\calU}{{\mathcal U}}
\newcommand{\calV}{{\mathcal V}}

\newcommand{\frakp}{\mathfrak p}

\newcommand{\bfa}{\mathbf{a}}

\newcommand{\ulom}{\underline{\omega}}
\newcommand{\uleta}{\underline{\eta}}
\newcommand{\ulJ}{\underline{J}}
\newcommand{\ulzeta}{\underline{\zeta}}
\newcommand{\ula}{\underline{a}}
\newcommand{\ulu}{\underline{u}}
\newcommand{\ulv}{\underline{v}}
\newcommand{\ulw}{\underline{w}}

\newcommand{\diag}{\mathrm{diag}} 
\newcommand{\fdiag}{\mathrm{fdiag}} 
\newcommand{\ff}{\mathrm{ff}} 
\newcommand{\GH}{\mathrm{GH}} 
\newcommand{\WH}{\mathrm{WH}} 
\newcommand{\IH}{\mathrm{IH}} 
\newcommand{\model}{\mathrm{model}} 
\newcommand{\gGH}{g_{\mathrm{GH}}} 
\newcommand{\redu}{\mathrm{red}} 
\newcommand{\LapGH}{\Delta_{g_{\GH}}} 
\newcommand{\TY}{\mathrm{TY}} 
\newcommand{\gTY}{g_{\TY}} 
\newcommand{\LapTY}{\Delta_{g_{\TY}}} 
\newcommand{\CA}{\mathrm{C}}

\newcommand{\gbfa}{g_{\bfa}} 
\newcommand{\Lapbfa}{\Delta_{\bfa}}

\newtheorem{theorem}{Theorem}
\newtheorem{proposition}{Proposition}
\newtheorem{corollary}{Corollary}
\newtheorem{lemma}{Lemma}
\theoremstyle{definition}
\newtheorem{definition}{Definition}
\theoremstyle{remark}
\newtheorem*{remark}{Remark}

\title{Tian--Yau metrics: \\
Fredholm theory, Hodge cohomology and moduli spaces} 

\author{Rafe Mazzeo \\ Stanford University \and Xuwen Zhu \\ Northeastern University}

\date{}

\begin{document}

\maketitle

\begin{abstract}
We study the natural geometric elliptic operators on a class of complete Riemannian manifolds which include
the 4-dimensional ALH* gravitational instantons and their higher dimensional Calabi-Yau analogues asymptotic to the model
Calabi Ansatz metrics. Some of these were initially constructed by Tian and Yau \cite{TY}, later by Hein \cite{Hein} and most
recently by Y.\ Chen \cite{YifanChen}, and we call these Tian-Yau spaces. They have played an important role in the degeneration 
theory of K3 metrics, cf.\ \cite{hsvz}, \cite{SunZhang2}, 
in the increasingly refined classification of gravitational instantons \cite{CJL}, \cite{LeeLin}, and other areas.  We show that these
elliptic operators can be analyzed using the $\bfa$-pseudodifferential calculus of Grieser and Hunsicker \cite{GH, GH2}, and
use this to determine the space of $L^2$ harmonic forms in the four-dimensional setting, as well as the refined asymptotic regularity
and  local deformation theory of ALH* structures. 
\end{abstract}

\section{Introduction}
Gravitational instantons are complete, noncompact hyperK\"ahler manifolds of dimension four with curvature tensor in $L^2$.  
They are fundamental examples of metrics with special holonomy, and as such, play an important role in several areas 
of mathematics and physics.   There are six types of gravitational instantons, going by the monikers ALE, ALF, ALG*, 
ALG, ALH* and ALH.  Through work of G.\ Chen and X.\ Chen \cite{ChenChen1, ChenChen2, ChenChen3}, and more 
recently Sun-Zhang \cite{SunZhang}, these constitute a complete classification of the possible asymptotic geometries
for this class. Each of these spaces is a singular torus fibration over an asymptotically conic (AC) base. The different 
types above correspond to different fiber and base dimensions and geometries. The fibers on ALE spaces are trivial,
so the spaces themselves are asymptotically conical. On ALF, ALG and ALH spaces, the fibers are tori of dimension $k = 1, 2$ and $3$, 
respectively, and the bases are AC spaces of dimension $4-k$. For each of these, the torus fibers (as Riemannian manifolds)
converge at infinity to some fixed flat torus of the same dimension. These four `regular types' are distinguished by 
having curvature decay like $r^{-2-\epsilon}$ for some $\epsilon > 0$, and for volume growth $r^{4-k}$, where
$r$ is distance from a fixed point in $M$. There are two exceptional cases: ALG* spaces are singular $T^2$ fibrations 
over a two-dimensional AC base, but where the conformal modulus of the torus fibers degenerates at infinity, while
ALH* spaces are singular nilmanifold fibrations over the half-line. ALG* spaces have curvature decay $1/r^2 \log r$ and 
volume growth of order $r^2$, while ALH* spaces have curvature decay $r^{-2}$ and volume growth of order $r^{4/3}$. 
There is a further classification within each asymptotic type, including a topological type determined by the middle
degree integer homology and intersection form, and more refined continuous moduli.

Fixing the asymptotic and topological type, the multi-dimensional moduli are understood only in certain cases. For example, 
Kronheimer \cite{Kr1, Kr2} famously classified ALE spaces of type ADE by realizing them as hyperK\"ahler reductions of finite 
dimensional affine spaces. Boalch's `modularity conjecture' speculates that all gravitational instantons are realized as 
gauge-theoretic moduli spaces (hence as hyperK\"ahler reductions of infinite dimensional affine spaces of connections and Higgs fields),
cf.\ the work of Cherkis and Kapustin \cite{ChK1, ChK2, ChK3}, Cherkis \cite{Cherkis},  Tripathy and Zimet \cite{TZ}, and more recently \cite{FMSW}.  
Gravitational instantons arise as rescaled pointed limit spaces in the degeneration theory of K3 surfaces and as bubbles in 
complex structure degeneration of Calabi-Yau manifolds, see \cite{SunZhang2}, \cite{hsvz} and also \cite{BG}. 
They also arise as mirrors in SYZ duality~\cite{CJL, CJL2}. 

Beyond this four-dimensional setting, there are various known families of higher dimensional Calabi-Yau metrics with asymptotics analogous 
to those for ALH* metrics and modelled on the higher dimensional Calabi Ansatz, see ~\cite{YifanChen}.  We mention also the metrics 
constructed in \cite{CollinsLi, CTY}, where the divisor at infinity is no longer smooth but has normal crossings.

With this preamble, we now explain the main goals of the present paper. We develop analytic tools for studying the natural elliptic operators on ALH* spaces and
their higher dimensional analogues. While some of this is already known, \cite{hsvz, SunZhang3, Carron}, the `uniform' approach here provides, 
for example,  Fredholmness of all these operators, both scalar and systems, between weighted Sobolev and H\"older spaces, as well as the associated 
regularity theory and mapping properties.  To these ends, we use the methods of microlocal analysis, specifically the calculus of $\bfa$-pseudodifferential
operators developed by Grieser and Hunsicker \cite{GH, GH2}, to construct parametrices for these elliptic differential operators. 
We then consider two key applications. The first is to identify the space of $L^2$ harmonic differential forms, also known as the Hodge
cohomology; this generalizes \cite{HHM}.  The second is to study the local deformation theory of ALH* spaces using Donaldson's 
method of hyperK\"ahler triples \cite{Donaldson}, see also \cite{FLS}. The key new issue involves the hyperK\"ahler 
deformations arising from variations of the geometric structure at infinity. 

The title of this paper is inspired by the fact that the first key examples of the types of metrics considered here were constructed
initially by Tian and Yau \cite{TY} in the late 1980's, and for this reason, we refer to the entire class of such metrics as Tian-Yau.

In the next section we review two different representations of their asymptotic models, the first a version of the Gibbons-Hawking ansatz
which is specific to the four-dimensional setting, and the second a model which holds in higher dimensions as well, called the Calabi Ansatz.  
We then review the key features of the Grieser-Hunsicker `\textbf{a}-calculus' of pseudodifferential operators and corresponding
elliptic parametrix construction. The rest of the paper uses these parametrices to establish Fredholm and regularity results, and
then goes on to study the Hodge cohomology and local deformation theory of ALH* spaces. 

\section{Asymptotic geometry of Tian-Yau metrics}
We now recall two useful models of Tian-Yau asymptotic geometry. The first, known as the Gibbons-Hawking Ansatz, 
specializes a remarkable representation formula for four-dimensional hyperK\"ahler metrics with triholomorphic $\bS^1$ symmetry; the
second, the Calabi Ansatz, is equivalent to Gibbons-Hawking in four dimensions, but also models certain asymptotic types of
Calabi-Yau metrics in higher (complex) dimensions. This material is gleaned from \cite{hsvz} and \cite{CJL}, cf.\ also \cite{Hein} and \cite{TY},
and we refer to these sources for more details.

\medskip

\noindent{\bf The Gibbons-Hawking Ansatz}.  This remarkable formula represents any $4$-dimensional hyperK\"ahler
metric with a triholomorphic circle action (i.e., an $\bS^1$ action which preserves each of the three complex structures). 
Let $\calU$ be an open set in $\RR^3$, or more generally, in any flat $3$-manifold with metric
$g_\calU$, and suppose that $M$ is a principal $\bS^1$ bundle over $\calU$.  Let $V$ be any positive harmonic function 
on $\calU$, and define the $1$-form $\alpha$ on $M$ by
\[
\star dV = d\alpha;
\]
the star and differential here are those on the flat $3$-manifold $\calU$. The harmonicity of $V$ is, of course, the integrability condition,
and $\alpha$ is well-defined if $\calU$ is simply connected. In terms of a linear coordinate $\theta$ on the $\bS^1$ fibers, then $d\theta + \alpha$ 
is a connection $1$-form on this principal bundle. Using these various data, we now write
\[
g = V g_{\calU} +  V^{-1} (d\theta + \alpha)^2.
\]
Such a metric is always hyperK\"ahler, in particular Ricci flat, and the obvious $\bS^1$ action rotating the fibers is triholomorphic.

Our interest is in the special case where $\calU$ is an open neighborhood in the flat half-cylinder $\bT^2 \times \RR^+_r$, where $\bT^2 = 
\RR^2/\Lambda$, and the circle bundle over $\calU$ restricts to each $\bT^2 \times \{r=\mbox{const.}\}$ as the flat $\bS^1$ bundle of degree 
$b$. In other words, this cross-section is the nilmanifold $Y = Y_b$, which is the total space of this circle bundle. 
Using Euclidean coordinates $(y_1, y_2)$ on $\RR^2$, and setting $w = y_2 + i y_1$, we assume that the lattice $\Lambda = \langle 1, \tau \rangle$,
where $\tau$ lies in the standard fundamental domain of the modular group.  With $r > C \geq 0$ as the linear coordinate on the
$\RR^+$ factor, then we take $V(y,r) = \beta \, r$ for any $\beta >0$ as a simple choice of positive harmonic function.  With this choice,
$dV = \beta\, dr$ and $\star dV = \beta \, dy_1 \wedge dy_2$.  However, $\star dV = d\alpha$ is also the curvature of the line bundle, which
indicates that we should take $\beta = 2 \pi b/ \mathrm{Im}\, \tau$.   Then write $\alpha = \beta \, y_1 dy_2$; this is not a valid expression
globally on $Y$, but we can regard it either as a local expression or as making sense globally on the universal cover $\widetilde{Y} = \mathrm{Nil}_3$.
On the other hand, $d\theta + \alpha$ is well-defined as the pushforward to $Y$ of a left-invariant $1$-form on $\mathrm{Nil}_3$,  cf.\ \cite{hsvz}.
For our purposes it suffices to use this local expression, which yields the metric 
\begin{equation}
\begin{aligned}
g_{GH}  & :=  \beta\, r (dr^{2}+|dy|^2) + (\beta \, r)^{-1} \alpha^2 \\
& = \beta\, r (dr^2 + dy_2^2 + dy_2^2)  + (\beta \, r)^{-1}  (d\theta+ \beta\, y_1dy_2)^{2}.
\end{aligned}
\label{GH}
\end{equation}

To expand on this, following \cite[pp.\ 58-61]{CJL} and using the notation there, this choice of $\alpha$ corresponds
to a choice of character $a:\Lambda \to \mathbb{S}^1$ which satisfies $a(\tau) = 1$, and a corresponding choice of Hermitian metric $h = e^{-\phi}$
with $\phi =\frac{\beta}{2} (y_{1}^{2}+y_{2}^{2})$.  We note that this metric $h$ is well-defined with respect to the Abel-Jacobi map which defines the line bundle 
on $\mathbb{T}^{2}$. To get the form for $\alpha$ above, we modify the coordinate on $\mathbb{S}^{1}$ fiber to be $\theta=\psi - \frac{\beta}{2} y_{1}y_{2}$.
(We refer to equation ~\eqref{e:CJLCoordinate} below for details.) 

This Gibbons-Hawking metric is the asymptotic model for the geometry of an end of an ALH* space in the following sense: if $g_{TY}$ is an ALH* 
Ricci-flat metric obtained by solving the Monge-Ampere equation on some quasiprojective smooth variety, then up to a diffeomorphism, $g_{TY}$ is 
asymptotic to this model metric up to an exponential error. In terms of the associated K\"ahler forms, 
$$
\nabla^{k}(\omega_{\TY}-\phi^{*}\omega_{\GH})=\calO(e^{-\delta r}),\ k = 0, 1, 2, \ldots, \ \ r \to \infty
$$
for some $\delta > 0$, see \cite[Prop 3.4(c)]{hsvz}.  

\medskip

\noindent{\bf The Calabi ansatz} 
There is an alternate model for the metrics on these ends due to Calabi. Unlike the Gibbons-Hawking formula, this Calabi ansatz
makes sense in any complex dimension $n \geq 2$, see~\cite{hsvz, SunZhang3}.   Fix a compact K\"ahler manifold $D$ of dimension
$n-1$ with trivial canonical bundle and $L$ an ample line bundle of degree $b$ over $D$. There exists a unique Ricci-flat metric $g_D$ on $D$, with K\"ahler form
satisfying $\omega_D^{n-1} = \frac{1}{2}i^{(n-1)^2} \Omega_D \wedge \overline{\Omega}_D$, where $\Omega_D$ is a holomorphic $(n-1)$-form
on $D$ satisfying $\int \frac12 i^{(n-1)^2} \Omega_D \wedge \overline{\Omega}_D = (2\pi c_1(L))^{n-1}$.  Let $h$ be the unique
Hermitian metric on $L$ with curvature form $-i \omega_D$, and consider the subset of $L$ consisting of all elements of length less
than $1$ with respect to $h$.  Then the Calabi ansatz is the K\"ahler form 
\[
\omega_\CA = \frac{n}{n+1} i \del \overline{\del} (- \log \|\xi\|_h^2)^{\frac{n+1}{n}}.
\]
This is Ricci flat, and in complex dimension $2$ yields an alternate asymptotic model for an ALH* Tian-Yau metric, up to errors decaying exponentially in $r = \|\xi\|_h$.

Using a radial coordinate $r$ on the fibers of $L$ and the connection $1$-form $\Theta$, we calculate that
\[
g_\CA=r^{\frac{2}{n}}(dr^{2}+g_{D}) + r^{\frac{2}{n}-2}\Theta^{2}.
\]
To match the methods here, we compactify this end by adding a boundary at $r = \infty$, or equivalently, by setting $x = 1/r$ and regarding $x=0$ as a 
boundary hypersurface, so that
\begin{equation}
g_\CA = x^{-\frac{2}{n}-4} dx^{2} + x^{-2/n} g_{D} + x^{2-2/n} \Theta^{2}=x^{-\frac{2}{n}+2} (x^{-6}dx^{2} + x^{-2} g_{D} + \Theta^{2}).
\label{C11}
\end{equation}
Setting $n=2$ again, this becomes
\begin{equation}
g_\CA = r (dr^2 + g_D) + r^{-1} \Theta^2 = \frac{dx^2}{x^3} + \frac{g_D}{x} + x \Theta^2.
\label{C1}
\end{equation}

Recent work of Y.\ Chen~\cite{YifanChen} considers existence and uniqueness of Ricci-flat metrics asymptotic to this Calabi Ansatz in higher dimensions. 

All of the analytic techniques used to obtain the Fredholm and mapping property results in this paper carry over to this higher dimensional 
setting as well, but for the sake of this presentation, we restrict most of the discussion to the complex two-dimensional case.  In particular, 
the two main applications of the analytic techniques here, the Hodge and deformation theorems, are proved only for four-dimensional
ALH* spaces.  The Hodge theory section generalizes readily to higher dimensions, but is computationally much more involved then.
However, there is no fundamental difficulty in generalizing the results here to higher dimensions, cf.\ \cite{HHM}. The deformation theory 
uses many ideas specific to this low dimension, so is likely to be more difficult (and less explicit) in higher dimensions. There is one 
higher dimensional problem taken up in this paper though: in Section 6 we prove an asymptotic regularity theorem for Ricci flat 
merics which are `weakly Calabi' (in the parlance of \cite{YifanChen}, but which is trivial when $n=2$).

\begin{remark}
In this entire paper, we consider metrics $g_{\TY}$ with this asymptotic Calabi (or Gibbons-Hawking) structure, and for simplicity, call
these Tian-Yau metrics. 
%of Tian-Yau type. By definition this is a metric asymptotic to the model Gibbons-Hawking metric
%up to some lower order error terms:  $g_{\TY} = g_{\GH} + h$, where $|h|_{\GH} = \calO( x^\mu)$ for some $\mu > 0$, along with corresponding decay for
%its covariant derivatives (at least up to some specified order).  As explained earlier, in certain cases it is reasonable to assume the much stronger condition
%that this error term decays exponentially, like $e^{-\mu/x}$. 
\end{remark}

\section{Fundamentals of geometry and analysis on $\bfa$-manifolds} %category and its Fredholm theory}
We now describe the geometric framework used here to analyze the operators associated to Tian-Yau and Calabi Ansatz metrics. 
Namely, we consider these metrics as special cases of the broader class of so-called $\bfa$-metrics, as defined by Grieser and Hunsicker
\cite{GH2}.  This general class, which includes and generalizes the metrics on $\mathbb Q$-rank $1$ ends of locally symmetric
spaces, is defined on the interior of any compact manifold with boundary $M$, where $\del M := Y_2$ is the total space of a double-fibration 
\begin{equation*}
\begin{tikzcd}
F \arrow[r, hookrightarrow] & Y_2 \arrow{d}{\pi_2}  \\ & Y_1
\end{tikzcd}
\qquad \mbox{where}\qquad
\begin{tikzcd}
F_1 \arrow[r, hookrightarrow] & Y_1 \arrow{d}{\pi_1} \\ & Y_0.
\end{tikzcd}
\end{equation*}
%\begin{equation*}
%\begin{aligned}
%F \hookrightarrow & B_2 \\ %\overset{\pi_2}{\longrightarrow} 
%& \downarrow \\ &B_1 %, \ \ \mbox{where} \ \  F_1 \hookrightarrow B_1 \overset{\pi_1}{\longrightarrow} B_0.
%\end{aligned}
%\end{equation*}
Using a defining function $x$ for $\del M$, an $\bfa$-metric is then one which is asymptotic (up to varying degrees of decay, as described 
below) to a metric of the form 
\begin{equation*}
g = \frac{dx^2}{x^{2(1 + a_1 + a_2)}} + \frac{ (\pi_1 \circ \pi_2)^* h_{Y_0}}{ x^{2(a_1 + a_2)}}  + \frac{ \pi_2^* h_{Y_1}}{x^{2a_2}} +  h_{F}.
\end{equation*}
Here $h_{Y_1}$ is a symmetric two-tensor which is required to restrict to a metric on each fiber $F_1$. The moniker $\bfa$ refers to the pair  $(a_1, a_2)$.  

%In our setting, there is only a single fibration, i.e., $B_2 = B_1 = \bT^2$, with fiber $F = \bS^1$, 
%and with `weights' $a_1 = a_2 = 1$.  

In the following, we restrict attention to a compact manifold with boundary $M$ (we usually write $M$ instead of $\overline{M}$) and assume that the
double fibration collapses to a simple fibration.  This is a special case of the double-fibration setup, where $Y_0 = \{\mathrm{pt.}\}$, $(F_{1}=)Y_1 = Y$ is
the base, $Y_2 = \del M$ and $F_2 = F$, with `weight' parameters are now $a_1 = a_2 = 1$. The point is that the metric $h_{Y_0}$ is now trivial, so the second term 
in $g$ is absent, and $h_{Y_1}$ is a nondegenerate metric on $Y$; thus our model metric takes the form
\begin{equation}
g = \frac{dx^2}{x^6} + \frac{\pi^* h_Y}{x^2} + h_F.
\label{gfamod1}
\end{equation}
We call this a simple $\bfa$-metric. Certain constructions below are carried out with this notation, but to simplify exposition in many specific calculations, we 
focus even further on the specific 
case where $Y = \bT^2$ and $F = \bS^1$.  

\subsection{The fundamental objects}
We now discuss some of the basic `flora and fauna' associated to a simple $\bfa$-structure on a manifold with boundary $M$.

The first order of business is to explain why Tian-Yau and Calabi Ansatz metric fit into this framework. In other words, we first
examine the degeneracy structure of an ALH* Laplacian.    Let us start with the local coordinate expression
\[
g = r(dr^{2}+dy_{1}^{2}+dy_{2}^{2}) + r^{-1}(d\theta+y_{1}dy_{2})^{2},
\]
which is a metric $(1,\infty) \times \bT^3 = (1,\infty)_r \times \bT^2_y \times \bS^1_\theta$. Here $(y_1, y_2)$ are Euclidean coordinates on $\bT^2$
and $y_1 dy_2$ is a local representation of the $1$-form $\alpha$.  This metric is of course not globally defined on $\bT^3$, but we are using
it here simply to illustrate the form of the degeneration.  Replacing $r$ by $x = 1/r$, this becomes
\begin{equation}
\gGH := \frac{dx^2}{x^5} + \frac{dy_1^2 + dy_2^2}{x} + x (d\theta + y_1 dy_2)^2,
\label{GHmetric}
\end{equation}
as previously recorded in \eqref{GH}; this has associated Laplace-Beltrami operator
\begin{equation}
\begin{split}
\Delta_{\gGH} &=r^{-1}\partial_{r}^{2} + r^{-1}(\partial_{y_{1}}^{2}+\partial_{y_{2}}^{2}) -2r^{-1} y_{1} \partial_{y_{2}}\partial_{\theta}  +(r+r^{-1}y_{1}^{2})\partial_{\theta}^{2} \\ 
& = x^{5}\partial_{x}^{2} + 2x^{4}\partial_{x}+ x (\partial_{y_{1}}^{2}+\partial_{y_{2}}^{2}) -  2x y_{1} \partial_{y_{2}}\partial_{\theta} + (x^{-1}+xy_{1}^{2})\partial_{\theta}^{2}. 
\end{split}
\label{e:delta}
\end{equation}
Multiplying by $x^{-1}$ leads to the conformally related $\bfa$-metric
\begin{equation}
g_{\bfa} = \frac{dx^2}{x^6} + \frac{dy_1^2 + dy_2^2}{x^2} + (d\theta + y_1 dy_2)^2.
\label{a-metric}
\end{equation}
We are not interested here in $\Delta_{g_{\bfa}}$, but in the structurally similar operator 
\begin{equation}
\begin{split}
x\LapGH  & =\left((x^{3}\partial_{x})^{2}+(x\partial_{y_{1}})^{2} + (x\partial_{y_{2}})^{2}+\partial_{\theta}^{2} \right) \\ 
& \qquad \qquad + \left(-x^{5}\partial_{x} -2x^{2}y_{1} \partial_{y_{2}}\partial_{\theta} + x^{2}y_{1}^{2}\partial_{\theta}^{2} \right). 
\end{split}
\label{delta2}
\end{equation}
The point is that it is more convenient to work with operators whose coefficients are smooth in this coordinate system. It suffices to study -- and
eventually construct a parametrix for -- $x \LapGH$ since this translates immediately to similar results for $\LapGH$ itself.  The key point is that $x\LapGH$ and $\Delta_{g_{\bfa}}$ are both finite sums of products of vector fields in the space
\[
\calV_{\XX} = \mathrm{span}_{\calC^\infty} \{ x^{3}\partial_{x}, x\partial_{y_{1}}, x\partial_{y_{2}}, \partial_{\theta}\}
\]
of `structure vector fields' for the $\bfa$-calculus.  Thus an arbitrary element of $\calV_{\XX}$ is one of the form
\[
V = a(x,y,\theta) x^3\del_x + b_1(x,y,\theta)x \del_{y_1} + b_2(x,y,\theta) x\del_{y_2} + c(x,y,\theta)\del_\theta,
\]
where $a, b_1, b_2$ and $c$ are all smooth up to $x=0$. This is a Lie algebra under the usual bracket which encodes the 
precise degeneracy structure of this whole theory. 

These $\bfa$-structures can also be defined invariantly.  As above, let $M$ be a compact manifold with boundary, with $\del M$ fibering over $Y$.
The space of $b$-vector fields $\calV_b(M)$ consists of all smooth vector fields on $M$ which are unconstrained in the interior, but which are tangent 
to $\del M$. If $x \in \calC^\infty(M)$ is any boundary defining function for $\del M$, then equivalently, a smooth vector field $V$ lies in $\calV_b$ 
if and only if $Vx = xF$ for some $F \in \calC^\infty(M)$, or as we typically write, $Vx = \calO(x)$.  An $\XX$ structure provides a refinement of $\calV_b$ 
which imposes different orders of vanishing of these vector fields according to their tangency to the fibers of $\del M$, but which also requires an
extension of the fibration of $\del M$ to the interior to some order. We phrase this in terms of a choice of an equivalence class of defining functions.  
Namely, fixing one defining function $x$, we set
\[
\calV_{\XX}(M) = \{ V \in \calV_b(M): \, \text{$V$ tangent to the fibers and}\ Vx = \calO(x^3)\}.
\]
This space is unchanged if we replace $x$ by another boundary defining function $\bar{x} = x + \calO(x^3)$; in other
words, this definition of $\calV_{\XX}$ is well-defined only in terms of a fixed $2$-jet of a defining function along the boundary. 
(Alternately, we could also define $\calV_{\XX}$ by defining an equivalence relation on $\calV_b(M)$.)  In any case, fixing a defining function $x$ (or
at least its equivalence class), the model $\bfa$-metric $g_{\bfa}$ in \eqref{a-metric} has the property that $g_{\bfa}(V,W)$ is smooth up to $x=0$
for any two vector fields $V, W \in \calV_{\XX}$. More generally, we can consider any metric $g$ on (the interior of) $M$ for which $g(V,W)$ is
smooth whenever $V, W \in \calV_{\XX}$. It is easy to see that $g$ must be a smooth combination of the basic (`singular') $2$-tensors
\[
\frac{dx^2}{x^6},\ \frac{dx \, dy_i}{x^4},\ \frac{dx\, d\theta}{x^3}, \frac{dy_i dy_j}{x^2},\ \frac{dy_i \, d\theta}{x},\ \mbox{and}\ d\theta^2.
\]

\begin{remark}
We are primarily interested in metrics which are asymptotic to the model form above, at least to higher order.  Indeed, we already noted
that general ALH* metrics are asymptotic to Gibbons-Hawking metrics up to terms which vanish like $e^{-\delta/x}$ for some $\delta > 0$.
Similarly, any Calabi-Ansatz metric (or rather, $x^{-1}$ times a Calabi Ansatz metric) is asymptotic to some model metric $g_{\bfa}$ up
to some higher polynomial order.  This becomes relevant when we study asymptotics of solutions, e.g., to equations like
$\Delta_{g_{\bfa}} u = 0$; such $\bfa$-harmonic functions then decompose asymptotically into (globally defined) components which
decay at very different rates.  This decoupling into different components is impossible to define unless the metric is asymptotic to
a model metric to sufficiently high order.  We discuss all of this in various places below.
\end{remark}

There are various other objects naturally associated to an $\bfa$-structure. These include the $\bfa$-tangent bundle ${}^{\bfa}TM$,
which has fiber at the point $p \in M$
\[
{}^{\bfa}T_pM = \calV_{\bfa}/ \mathcal I_p \calV_{\bfa},
\]
where $\mathcal I_p$ is the ideal of smooth functions on $M$ vanishing at $p$. If $p \in \mathrm{int}\, M$, then ${}^{\bfa}T_p M$ is canonically
identified with $T_p M$, but if $p \in \del M$, then we can realize this fiber as follows: expanding any $V \in \calV_{\bfa}$ as $V = a x^3\del_x + 
b_1 x\del_{y_1} + b_2 x \del_{y_2} + c \del_\theta$, where $a,b_j, c \in \calC^\infty(M)$, then the residue class of $V$ in the quotient above 
is identified with the $4$-tuple $(a(p), b_1(p), b_2(p), c(p))$. In other words, $x^3\del_x, x\del_{y_1}, x \del_{y_2}$ and $\del_\theta$ constitutes a 
smooth local basis of sections of ${}^{\bfa} TM$ near $p$. Similarly, the dual bundle, the $\bfa$-cotangent bundle ${}^{\bfa}T^*M$ is spanned locally 
by the smooth basis of sections
\[
\frac{dx}{x^{3}}, \ \frac{dy_{1}}{x}, \frac{dy_{2}}{x}, \ d\theta.
\]
A general $\XX$-metric on $M$ is a smooth positive-definite section of $\mathrm{Sym}^2({}^{\bfa}T^*M)$.

An $\bfa$-differential operator $L$ is one which can be written as a locally finite sum of elements of $\calV_{\XX}$:
\[
L = \sum_{|J| \leq m}  V_{j_1} \ldots V_{j_m}.
\]
The space of all such operators is denoted $\Diff^*_{\bfa}(M)$. An operator of this type is $\bfa$-elliptic if it is an `elliptic combination' of 
a spanning basis of sections for $\calV_{\bfa}$.  In our $4$-dimensional setting, a second order operators $L$ is $\bfa$-elliptic if it is a positive 
definite quadratic form in the vector fields $x^3\del_x, x\del_{y_1}, x\del_{y_2}, \del_\theta$.  There is an invariant way to define the $\bfa$-symbol 
${}^{\bfa}\sigma_m(L)$ of any operator $L \in \Diff^m_{\bfa}(M)$; this is a smooth function on ${}^{\bfa}T^*M$ which is a homogeneous polynomial 
of order $m$ on each fiber.  

As noted earlier, these $4$-dimensional statements extend immediately to general dimensions. 

\subsection{Function spaces}\label{ss:space}
There are various function spaces which will be useful below. Some of these are ones on which $\bfa$-operators act most naturally, while
others are needed to formulate optimal regularity statements for solutions to $\bfa$-elliptic equations.   We define the spaces below only
for scalar functions or distributions, but it is clear that all of the definitions below extend immediately to sections of arbitrary bundles over $M$
which are smooth up to $\del M$. 

\medskip

\noindent {\bf $\bfa$-Sobolev spaces} Fix any $\bfa$-metric $g$ on $M$, with associated volume form $dV_g$.  Now define, for any $k \in \mathbb N$, 
the $\bfa$-Sobolev space of order $m$:
\[
H^k_{\bfa}(M) = \{u: V_1 \ldots V_\ell u \in L^2(M, dV_g)\ \ \forall \ V_{1}, \ldots, V_\ell \in \calV_{\bfa}\ \mbox{and}\ \forall\ \ell \leq k\}.
\]
We can then define $H^{-k}_{\bfa}(M)$ by duality, and the spaces of all real orders $H^s_{\bfa}(M)$ by interpolation. If $\mu \in \RR$, we also define the weighted Sobolev spaces 
\[
x^\mu H^s_{\bfa}(M) = \{ u = x^\mu v: v \in H^s_{\bfa}(M) \}.
\]

It is immediate that if $L \in \Diff^m_{\bfa}(M)$, then
\[
L: x^\mu H^{s+m}_{\bfa}(M) \longrightarrow x^\mu H^s_{\bfa}(M)
\]
is bounded for any $s, \mu \in \RR$. 

\medskip

\noindent {\bf $\bfa$-H\"older spaces}
As above, fix any $\bfa$-metric $g$. We may then define the associated `geometric' H\"older spaces: if $B_1(p)$ is the unit $g$-ball around 
$p \in \mathrm{int}\, M$, then 
\[
\calC^{0,\alpha}_{\bfa}(M) = \{u \in \calC^0(\mathrm{int}\, M): \sup_{p \in \mathrm{int}\, M} \sup_{q_1, q_2 \in B_1(p)} 
\frac{ |u(q_1) - u(q_2)|}{\mathrm{dist}_g(q_1, q_2)^\alpha} < \infty \}. 
\]
There is also a local coordinate description which follows by noting that the sets
\[
\{(x,y,\theta):  |x-x_0| \leq \frac12 x_0^3,\ |y-y_0| \leq \frac12 x_0,\ |\theta - \theta_0| \leq \frac12 \}
\]
is comparable to $B_1( (x_0, y_0, \theta_0)$, uniformly in $x_0 > 0$.  The H\"older spaces of higher order are 
\[
\calC^{k,\alpha}_{\bfa}(M) = \{u \in \calC^{0,\alpha}_{\bfa}(M): V_1 \ldots V_\ell u \in \calC^{0,\alpha}_{\bfa}(M)\ \ \forall\ V_j \in \calV_{\XX}\ \mbox{and}\ \ell \leq k\}.
\]
Finally, 
\[
x^\nu \calC^{k,\alpha}_{\bfa}(M) = \{u = x^\nu v: v \in \calC^{k,\alpha}_{\bfa}(M) \}.
\]

Once again, it is clear that if $L \in \Diff^m_{\bfa}(M)$, then
\[
L: x^\nu \calC^{k+m,\alpha}_{\bfa}(M) \longrightarrow x^\nu \calC^{k,\alpha}_{\bfa}(M)
\]
for any $\nu \in \RR$ and $k \in \NN_0$. 

\medskip

\noindent {\bf Conormality and polyhomogeneity}   We now review the important classes of `completely smooth' functions on the manifold $M$.
In fact, we describe all this in a slightly more general setting. In the next section we shall define the class of $\bfa$-pseudodifferential operators in 
terms of the precise regularity of their Schwartz kernels. These kernels are distributions on a certain blowup of $M \times M$, as described in that 
section, which are ``fully smooth'' on this manifold with corners.  

The most obvious (but naive) spaces of fully smooth functions on $M$ are the intersections $\bigcap_s x^\mu H^s_{\bfa}(M)$ or 
$\bigcap_k x^\nu \calC^{k,\alpha}_{\bfa}(M)$. However, a moment's thought shows that while elements in either of these spaces are $\calC^\infty$ in the 
interior of $M$, there is very little control of their regularity at $\del M$.  Furthermore, these are defined on $M$, but not on any possible blowup of $M \times M$.

Instead, suppose that $X$ is an arbitrary compact manifold with corners up to codimension $k$. This means that near any point $p \in X$, there is
a local diffeomorphism from a neighborhood of $p$ to a neighborhood of $0$ in the positive orthant $(\RR^+)^\ell \times \RR^{n-\ell}$ for some
$\ell \leq k$ (we say then that $p$ lies on a corner of codimension $\ell$ if this diffeomorphism maps $p$ to $0$).  
We next define the space $\calV_b(X)$ of $b$-vector fields on $X$ to consist of all smooth vector fields on the closed manifold $X$ which are unconstrained
in the interior and which lie tangent to all boundaries and corners. Using local coordinates $(x_1, \ldots, x_\ell, y_1, \ldots, y_{n-\ell})$ adapted to the
orthant, $\calV_b$ is spanned over $\calC^\infty(X)$ by the local basis of sections $x_i \del_{x_j}$, $1 \leq i, j \leq \ell$, and $\del_{y_s}$, $s = 1, \ldots, n-\ell$. 

\begin{definition}  We say that a function $u$ lies in the space of conormal functions of order $\nu \in \RR$, $\calA^\nu(X)$, if
\[
| V_1 \ldots V_\ell u  | \leq C x_r^{\nu}\ \ \mbox{for any}\ V_i \in \calV_b(X)\ \ \mbox{and for all}\ \ell \in \NN_0.
\]
The constant $C$ obviously depends on $\nu$, the $V_i$ and on $u$, and this bound holds near each boundary face $H_r = \{x_r = 0\}$.  
We can, of course, choose different weights $\nu_r$ for each boundary face $H_r$. 
\end{definition}
Any conormal function is $\calC^\infty$ in the interior of $X$ and has full tangential regularity along each boundary face. However, the fact that
we are differentiating with respect to $b$-vector fields means that $\calA^\nu$ contains such functions as $x_i^\gamma$ for 
any $\gamma$ with $\mathrm{Re}\, \gamma \geq \nu$, as well as $x^{\gamma} (\log x)^p$ for $\mathrm{Re}\, \gamma > \nu$ and any $p\in \RR$. 

It is much more convenient to work with an even smaller class of fully smooth functions, namely those which are not only conormal, but which
admit asymptotic expansions at each boundary face.   An index set $E$ is a discrete subset $\{(z_j, k_j)\} \subset \CC \times \NN_0$ such
that $\mathrm{Re}\, z_j \to +\infty$ as $j \to \infty$ and such that any vertical slab $\mathrm{Re}\, z \in (A,B)$ contains at most finitely
many of the $z_j$.  
\begin{definition}
A function $u(x,y)$ on a manifold with boundary $X$, where $x$ is a boundary defining function for $\del X$, is called a polyhomogeneous function 
with index set $E$ if it is an element of some $\calA^\nu(X)$ such that 
\[
u \sim \sum_{(z_j, k_j) \in E} \sum_{\ell=0}^{k_j} a_{j\ell}(y) x^{z_j} (\log x)^\ell,
\]
where the coefficient functions $a_{j\ell}$ are $\calC^\infty$ on the boundary $\{x=0\}$.  This asymptotic expansion is meant in the classical sense, and
holds even when differentiated any number of times. We write $u \in \calA^{E}_{\phg}(X)$. 

Next, suppose that $\{H_i, i = 1, \ldots , s\}$ is an enumeration of the boundary hypersurfaces of a manifold with corners $X$. Fix, for each $i$, an
index set $E_i$, and write $\calE = (E_1, \ldots, E_s)$; this is called an index family for $X$. A function $u$ on $X$ is called polyhomogeneous with index 
family $\calE$ if it is conormal with some choice of weights and, near each $H_i$, $u$ admits an expansion with index set $E_i$, with product type
expansions at all corners.
\end{definition}

It is immediate that $x\LapGH$ and $\Delta_{g_{\bfa}}$ act on $\calA^\mu(M)$ for any $\mu$.  These operators also preserve polyhomogeneity,
but typically alter the index sets. 

We refer to \cite{Ma-edge1} for more careful descriptions of these spaces, as well as a number of basic facts and results about them. 

\section{Parametrices for $\bfa$-elliptic operators}
We now turn to the main analytic part of the paper. The goal here is to show that the natural elliptic operators associated to an 
ALH* metric fit into the existing framework of the Grieser-Hunsicker `$\mathbf{a}$-calculus' \cite{GH, GH2}.   The advantage
of this approach is that these methods are flexible, and apply to such operators coming from both the warped product
model metric and the ones asymptotic to these models. These methods also apply easily to systems such as Dirac-type
operators and Hodge Laplacians. In addition, parametrix methods provide an easy way to translate $L^2$-based mapping properties
and Fredholm theory to analogous results on weighted H\"older spaces. Finally, parametrices yield refined details about the 
asymptotic structure of solutions to these equations.  We note, of course, that more classical methods involving separation
of variables and perturbation arguments provide a simpler approach to the Fredholm theory for scalar equations, see \cite{SunZhang3}.

The parametrix methods explained below are part of the general program of geometric microlocal analysis. As already seen in \S 3, elliptic operators 
associated to ALH* metrics are, up to a factor of $x$, degenerate elliptic operators on the compactification the ALH* space $M$
obtained by adding the nilmanifold $Y$, which is a nontrivial circle bundle over $\bT^2$, as the boundary at infinity.  We have
shown that if $\Delta_g$ is the Laplace operator  for a metric $g$ of this type, then $x\Delta_g$ lies in the
space $\Diff^2_{\bfa}(M)$ of $\bfa$-operators of order $2$; moreover, this operator is $\bfa$-elliptic.  The main idea
is to define a `calculus' of $\bfa$-pseudodifferential operators $\Psi^*_{\bfa}(M)$, in which $\Diff^*_{\bfa}(M)$ lies as a distinguished
subalgebra. This calculus is sufficiently large to contain parametrices of elliptic elements of $\Diff^*_{\bfa}(M)$. The term `calculus'
is meant to indicate that $\Psi^*_{\bfa}$ is not quite an algebra; composition is not defined for certain pairs of elements. The
obstruction, however, is a simple one related to integrability, and is easy to check in practice.   These $\bfa$-pseudodifferential
operators are characterized by the regularity properties of their Schwartz kernels, which are defined to be polyhomogeneous
functions on a certain blowup (or resolution) $M^2_{\bfa}$ of the product $M^2$, along with an extra classical singularity
along the diagonal. Indeed, this diagonal singularity is the `pseudodifferential part'; its other properties correspond 
to `far field' asymptotic behaviour of these Schwartz kernels away from the diagonal.   The point of blowing up $M^2$
to $M^2_{\bfa}$ is that this provides the most transparent way to keep track of these various asymptotic regimes.

We shall be describing all of this in our particular setting, where $M$ is four-dimensional, with boundary $Y$ the total
space of the $\bS^1$ fibration over $\bT^2$, where $y = (y_1, y_2)$ are flat coordinates on the torus and $\theta$ is the 
standard coordinate on the circle. There is very little that is special to this case, and the following material can easily
be adapted to where $y$ and $\theta$ are multi-dimensional variables. We take this path to keep the exposition as concrete as possible. 

There are, in fact, two intermediate types of degeneracies, with pseudodifferential calculi, which we also need and recall below.
These are the $b$-calculus $\Psi^{*}_b$ and cusp (or $c$-) calculus $\Psi^*_c$.  We consider the operators
in these two calculi in our setting as acting on reduced spaces, namely $[0, \epsilon_0)_x$ and $[0,\epsilon_0)_x \times \bT^2_y$, respectively.
This comes about because we are letting them act on the finite dimensional spaces of functions (or forms, etc.) which
restrict to be harmonic on the $\bS^1$ fibers (for the $c$-calculus) or on the $\bS^1$ fibers and on $\bT^2$ (for the $b$-calculus).
This will be explained more carefully below. The structure vector fields in these two cases, and on these reduced spaces, are
\begin{equation*}
\begin{aligned}
\calV_b([0,\epsilon_0) & = \mathrm{span}_{\calC^\infty} \{ x \del_x \}, \ \  \mbox{and} \\ 
\calV_c([0,\epsilon_0) \times \bT^2) & = \mathrm{span}_{\calC^\infty} \{ x^2 \del_x, \del_{y_1}, \del_{y_2} \}.
\end{aligned}
\end{equation*}
The associated pseudodifferential calculi, $\Psi^*_b( [0, \epsilon_0))$ and $\Psi^*_c( [0,\epsilon_0) \times \bT^2)$, are defined by
requiring the Schwartz kernels of their elements to live in the $b$-double space $[0,\epsilon_0)^2_b$ and the $c$-double
space $( [0,\epsilon_0) \times \bT^2)^2_c$, respectively, with polyhomogeneous expansions on all boundary faces. 
These operators can in fact be regarded as acting on `partially harmonic' functions on $M$, and thus can be regarded
as having Schwartz kernels on double spaces $M^2_b$ and $M^2_c$.  The blowups for these will, as for $M^2_{\bfa}$, be reviewed in
the next subsection.  In any case, we are really recalling three separate calculi: $\Psi^*_{\bfa}(M)$, $\Psi^*_c(M)$ and $\Psi^*_b(M)$;
the eventual parametrix will be a sum of three operators, one in each of these spaces. 

An important feature of each of these calculi, as distinct from the classical case on compact manifolds, is that there is a hierarchy
of symbol mappings, one for the singularity along the diagonal and others corresponding to the leading asymptotic
coefficients of the Schwartz kernel on the `front faces' of $M^2_{\bfa}$, i.e., the boundary hypersurfaces which intersect
the diagonal. These `boundary symbols' are simply the leading asymptotic of the operator.  One can prove Fredholmness only if all of these
symbol mappings are invertible, not just the one for the diagonal singularity!  An operator for which all of these symbol mappings
are invertible is called {\it fully elliptic}. 

In this section we first construct the spaces $M^2_b$, $M^2_c$ and $M^2_{\bfa}$ (these are all closely related, so it is not really three
separate constructions) and then define the spaces of $b$-, $c$- and $\bfa$-pseudodifferential operators. 
At that point we can explain the steps in the parametrix construction for fully elliptic $\bfa$-differential operator, focusing specifically 
on $x \LapTY:= \calL_{\bfa}$ for simplicity, where $g$ is a Tian-Yau ALH* metric.   The entire construction below can be adapted easily
for closely related elliptic operators, for example $x^{1/2} D_{g_{\TY}}$, where $D_{g_{\TY}} = d + \delta$ is the Hodge-de Rham operator for an ALH* space. 

The entire scheme of these definitions and parametrix constructions has been carried out in many other settings and for many other types of degeneracy structures.
Each setting has its own peculiarities, but the overall structure of how to proceed is basically the same in each case. 

For expository purposes, we first treat the simpler case where the fibration of $Y$ over $\bT^2$ is trivial, which is simpler from a geometric
point of view, and then turn to explaining the modifications necessary to treat the general twisted case. 

\subsection{Blowups and pseudodifferential operators}
Consider a $4$-dimensional space $M$ endowed with an $\bfa$-structure and $\bfa$-metric $g$.  We work in a product neighborhood 
of the boundary $\calU = [0,\epsilon)_0 \times Y$ with coordinates as above, and as a very first step, consider the case where the 
fibration is trivial, so that $Y = \bT^{2}_{y}\times \bS^{1}_{\theta}$. Then, at leading order, 
\[
\gbfa=\frac{dx^{2}}{x^{6}} + \frac{g_{\bT^{2}}}{x^{2}} + g_{\bS^{1}},
\]
though more generally we assume only that $g$ is  asymptotic to this model form. The Laplacian is 
\begin{equation}\label{e:prdLap}
\Lapbfa= (x^{3}\partial_{x})^{2} + x^{2} (\partial_{y_{1}}^{2}+\partial_{y^{2}}^{2}) + \partial_{\theta}^{2}. 
\end{equation}
When the fibration is nontrivial, we recall the Gibbons-Hawking model, which locally in $\bT^2$ takes the form
\begin{equation}
\begin{split}
g_{GH} & = x^{-1}(\frac{dx^{2}}{x^{4}} + dy_{1}^{2}+dy_{2}^{2})  + x (d\theta + y_{1}dy_{2})^{2}, \\ 
& = x \left(\left(\frac{dx^{2}}{x^6}+  \frac{g_{\bT^2}}{x^2}  + g_{\bS^1}\right) + \left(y_{1}^{2} dy_{2}^{2} + 2 y_1 d\theta dy_2\right) \right),
\end{split}
\label{e:h}
\end{equation}
so $\gbfa = x^{-1}\gGH$ is of $\bfa$-type.  Unlike the product case, this model metric has lower order cross-terms. For later reference, 
we denote by $h$ the tensor in the final set of parentheses in \eqref{e:h}. The Laplace operator for this metric is 
\begin{equation}
\begin{aligned}
\LapGH & = x^{-1} ( ((x^3\del_x)^2-x^{5}\partial_{x} +  &  x^2 (\del_{y_1}^2 + \del_{y_{2}}^2)+\del_{\theta}^2)  \\
& \qquad +  ( -2x^{2}y_{1} \partial_{y_{2}}\partial_{\theta} + x^{2}y_{1}^{2}\partial_{\theta}^{2} ) )\\
& = x^{-1} \calL_{\bfa}.
\end{aligned}
\label{lagnf}
\end{equation}
The operator $\calL_{\bfa} = x \LapGH$ is an elliptic $\bfa$-operator.  Here it is convenient to use the generating set of vector fields
\begin{equation}\label{e:vector}
x^{3}\partial_{x}, x\partial_{y_{1}}, x(\partial_{y_{2}}-y_{1}\partial_{\theta}), \partial_{\theta}
\end{equation}
in place of $\{x^3\del_x, x\del_{y_1}, x\del_{y_2}, \del_\theta\}$. 

\bigskip

\noindent{\bf Blowups} 
% First recall that $X$ is a manifold with corners if, for any $q \in X$, there exists a local coordinate system $(x_1, \ldots, x_\ell, y_1, \ldots, y_m)$ centered
% at $q$ with each $x_j \in [0,\epsilon)$ and $y_i \in (-\epsilon, \epsilon)$; we then say that $p$ lies on a corner of codimension $\ell$, namely
% the intersection of the boundary hypersurfaces $H_j = \{x_j = 0\}$, $j = 1, \ldots, \ell$.  

Suppose that $X$ is a manifold with corners. A submanifold $Y \subset X$ is called a $p$-submanifold if there exists a coordinate system adapted
to the orthant structure (as in the previous section) in some neighborhood $\calU$ around each $q \in Y$ such that 
$\calU \cap Y = \{x_1 = \ldots = x_{\ell'} = y_1 = \ldots y_{m'} = 0\}$ for some $\ell' \leq \ell$ and $m' \leq m$.  Thus $X$ has a product
decomposition around $q$ of the form $(Y \cap \calU) \times [0,\epsilon)^{\ell-\ell'} \times (-\epsilon, \epsilon)^{m-m'}$. 

If $Y$ is a $p$-submanifold of $X$, we define the (normal) blowup $[X; Y]$ as the disjoint union of $X \setminus Y$ and $SN^+ Y$, the inward-pointing
spherical normal bundle of $Y$ (inward pointing is only relevant if $Y$ lies in a boundary of $X$), with the obvious topology and smooth structure 
generated by the lifts of all smooth functions on $X$ and
the spherical normal coordinates around $Y$.  This blowup has a new boundary face equal to the set of all spherical
normal vectors; this fibers over $Y$ with fibers some spherical orthant (of dimension equal to the codimension of $Y$ in $X$). 

Any vector field $V$ on $X$ lifts to a vector field on $[X;Y]$; this lift is nonsingular if and only if $V$ is tangent to $Y$. 
Blowup has the effect of making such tangent vector fields less degenerate. Our goal here is to make each of the vector fields 
$x^3\del_x, x \del_{y_1}, x\del_{y_2}$ and $\del_\theta$ nondegenerate.  (These are initially vector 
fields on $M$, but we regard them as vector fields on $M^2$ by letting them act on the `left', i.e., first factor of $M$.) 

%The coordinates
%$(x,y_1, y_2, \theta)$ are duplicated to obtain a full set of coordinates $(x,\tilde{x}, y_1, \tilde{y}_1, y_2, \tilde{y}_2, \theta, \tilde{\theta})$ on $M^2$. 

First define the $b$-double space
\begin{equation}
M^2_b = M \times_b M ;= [M^2;  (\del M)^2_o],
\label{bbup}
\end{equation}
where $(\del M)^2_o$ consists of the components of $(\del M)^2$ which are the double products of each of the connected 
components of $\del M$ (thus we omit products of different components of $\del M$).  The closure of the lift of the interior of the diagonal 
$\diag \subset M^2$ is denoted $\diag_b$.  The `front face' $\ff_b$ is the new (and possibly disconnected) boundary face created by this blowup. 

Next define the $c$-double space ($c$ stands for cusp) 
\begin{equation}
M^2_c = M\times_c M := [ M^2_b ; \ff_b \cap \diag_b].
\label{cuspbup}
\end{equation}
This has lifted diagonal $\diag_c$, a front face $\ff_c$ created by the new blowup, as well as (the lift of the) $b$-front face, which we might
call the `intermediate' front face of $M^2_c$. It plays a relatively unimportant role when working with fully elliptic operators, as we discuss below.
Finally, define the full $\bfa$-double space
\begin{equation}
M^2_{\bfa} = M \times_{\bfa} M := [M^2_c; \fdiag_c],
\label{bfabup}
\end{equation}
where $\fdiag_c$ is the lift (in the sense above) of the fiber diagonal of the boundary $\{x = \tilde{x} = 0, y_1 = \tilde{y}_1, y_2 = \tilde{y}_2\}$ in $M^2$. 
We use here a duplicated set of coordinates $(x, y, \theta, \tilde{x}, \tilde{y}, \tilde{\theta})$ on $M^2$. The lifted diagonal of $M^2_{\bfa}$ is 
denoted $\diag_{\bfa}$;  there is a front face $\ff_{\bfa}$, and two intermediate front faces which are the lifts
of the $b$- and $c$- front faces; as before, these play an insignificant role when working with fully elliptic operators. 

\bigskip

\noindent{\bf Lifts of vector fields}

We now compute the lifts of the basic $\bfa$-vector fields to these spaces.  This is most easily done in terms of projective coordinates for
each blowup.  Thus, a local coordinate system on $M^2_b$ is obtained by setting $s = x/\tilde{x}$, so that the full coordinate system
is $(s, y_1, y_2,\theta, \tilde{x}, \tilde{y}_1, \tilde{y}_2, \tilde{\theta})$. The front face $\ff_b$ is $\{\tilde{x} = 0\}$. 
These are valid coordinates away from the lift of the face where $\tilde{x} = 0$. 
In terms of these, the $\bfa$-vector fields are lifted to
\[
x^3\del_x = \tilde x^2 s^3 \del_s, \ \ x \del_{y_j} = \tilde{x} s \del_{y_j}, \ \ \del_\theta.
\]
The first of these is slightly less degenerate near $\ff_b \cap \diag_b$ because $s = 1$ there.  However, they still degenerate, so we proceed further.
To compute their lifts to $M^2_c$, set $s' = (s-1)/\tilde{x}$, keeping all the other coordinates the same, to get 
\[
\tilde{x}^2 s^3\del_s = \tilde{x} (1 + \tilde x s')^3 \del_{s'}, \ \ \tilde{x} s \del_{y_j} = \tilde{x} (1 + \tilde{x} s') \del_{y_j},  \ \ \del_\theta. 
\]
As before, the first of these is one step less degenerate near $\ff_c$, but now, the first three of these vector fields vanish at the same rate at $\ff_c$.

For the final lift to $M^2_{\bfa}$, we use projective coordinates $S = s'/\tilde{x}$, $Y_j = (y_j-\tilde{y}_j)/\tilde{x}$, which yields the lifts
\[
\tilde{x} (1 + \tilde{x} s')^3 \del_{s'} = (1 + \tilde{x}^2 S)^3 \del_S, \ \  (1 + \tilde{x}^2 S) \del_{Y_j}, \ \ \del_\theta.
\]
Thus at the newest front face $\ff_{\bfa}$, where $\tilde{x} = 0$, these lifts are the fully nondegenerate vector fields 
$\del_S$, $\del_Y$, $\del_\theta$.  Note that at each of these steps, on the respective front faces, $s$, $s'$ and $(S,Y)$ are global variables on 
$(0,\infty)$, $\RR$ and $\RR^3$, respectively, and not just local coordinates. 

\bigskip

\noindent {\bf Pseudodifferential operators}

Associated to each of the three degeneracy structures above is a calculus of pseudodifferential operators. The moniker `calculus' (rather than, e.g.\ algebra)
indicates two features. First, not every pair of elements can necessarily be composed; the obstruction is, however, a very simple one related to
whether certain integrals are convergent. Second, these spaces of operators come equipped with extra structure, including a hierarchy of symbol
maps, and the corresponding subcalculus of residual operators; these symbols are the key tools in constructing parametrices for fully elliptic elements.
We shall be quite brief in the description below because it is amply documented elsewhere: see \cite{Ma-edge1} or \cite{APS} for the case of $b$-operators,
and \cite{GH2} for $c$- and $\bfa$-operators.

We characterize the elements of $\Psi^*_{\sharp}(M)$, $\sharp = b,\, c$, or $\bfa$, by the geometric structure of their Schwartz kernels. 
If $A$ is an operator in either of these three classes, then its Schwartz kernel $K_A$ by definition lifts to $M^2_{\sharp}$ to have a particularly
simple geometric structure:  
\begin{itemize} 
\item[i)] first, it has a `classical' pseudodifferential singularity along the lifted diagonal $\diag_{\sharp}$ which is
uniform along this $p$-submanifold up to the front face, up to a factor of $\rho^q$ for some $q \in \RR$.  Said slightly differently,
we can localize $K_A$ by multiplying by a cutoff function $\chi \in \calC^\infty(M^2_{\sharp})$ which has support in a small neighborhood
of $\diag_{\sharp}$ and which vanishes at all boundary faces except $\ff_{\sharp}$. Then taking the Fourier transform of $\chi K_A$ 
transversely to $\diag_{\sharp}$, one obtains a classical symbol $a(z,\zeta)$, which is of the form $\rho^q \tilde{a}(z,\zeta)$, where $\tilde{a}$
is a standard symbol on the conormal bundle of $\diag_{\sharp}$ which is smooth up to the front face, and $\rho$ is a defining function
for $\ff_\sharp \cap \diag_\sharp$;
\item[ii)] second, $(1-\chi)K_A$ is polyhomogeneous on $M^2_\sharp$.  The exponents which appear in its expansions at each boundary
face are regulated by an index set, so the precise class of polyhomogeneous functions in which it is contained is determined by
an index family $\calE$, which is a union of index sets, one for each boundary face of the double space. The index set associated
to the front face $\ff_{\sharp}$ is usually one of the form $\{(q+j, 0), j \in \NN_0\}$, which matches the behavior of the symbol
along the lifted diagonal. 
\end{itemize}
Thus to record all of this information, one considers operators in the subclasses $\Psi^{m, \calE}_{\sharp}(M)$, where $m$ is the
order of pseudodifferential singularity along the diagonal and $\calE$ records the index sets for the polyhomogeneous expansions
at each face.

One of the main structural theorems is that composition of two such operators, $A \circ B$, is again an operator of the same type
(with diagonal singularities and index sets composing accordingly). The obstruction to composition is merely the issue of whether the
integral
\[
\int_M  K_A(z, z') K_B(z', z'')\, dz',
\]
converges, which depends on the sum of the rates of blowup or decay of $K_A$ and $K_B$ as $x' \to 0$. 

\subsection*{Parametrix construction}

We now describe the parametrix construction for a fully elliptic $\bfa$-differential operator. The key examples we have in mind 
are 
\[
\calL_{\bfa} = x\LapTY
\]
and $x^{1/2} D_{g_{\TY}}$ where $\gTY$ is a Tian-Yau metric and $D_{g_{\TY}} = d + \delta$ is its 
Hodge-de Rham operator.  For expository purposes, we carry out the construction for the scalar Laplacian only, but all steps below 
translate easily to this other operator (and for certain more general $\bfa$-operators as well). 

\medskip

\noindent{\bf Decomposition of $\LapTY$}   

The first part of this construction is to recognize that near $\del M$, the Laplacian admits a functional decomposition into three fully 
elliptic summands, the first a $b$-operator, the second a $c$-operator and the third an $\bfa$-operator.   This decomposition is obtained
by considering the action of $\LapTY$ on functions in a collar neighborhood $\calU$ of $\del M$ which are either harmonic or 
perpendicular to the harmonic functions on $\bS^1$ factor, and then, for the $\bS^1$-harmonic part, restricting further to functions 
which are harmonic or perpendicular to harmonic on the $\bT^2$ factor.  (Of course, these decompositions are more interesting
for the Hodge Laplacian on forms, but we stay with functions to keep the exposition straightforward.)

More carefully, consider the space of fiber harmonic functions and corresponding orthogonal projection
\begin{equation*}
%\begin{aligned}
L^{2}H^*_{\bS^1}(\calU):=\{f \in L^2(\calU): \Delta_{\bS^{1}}f=0\},  \quad \Pi_{\theta}: L^2(\calU) \rightarrow L^{2}H^*_{\bS^1}(\calU).
%\end{aligned}
\end{equation*}
Setting $\Pi_\theta^\perp = I - \Pi_\theta$, then over $\calU$,
\[
\LapTY = \Pi_\theta \LapTY \Pi_\theta + \Pi_\theta^\perp \LapTY \Pi_\theta^\perp + E_1,
\]
where $E_1 = \Pi_\theta \LapTY \Pi_\theta^\perp + \Pi_\theta^\perp \LapTY \Pi_\theta$ contains the `cross-terms'. 
This extra term $E_1$ vanishes to whatever order the boundary fibration of $Y$ over $\bT^2$ extends into the interior. 

To continue, we regard $\Pi_\theta \LapTY \Pi_\theta$ as acting on sections of the bundle $\calH_\theta$ of $\bS^1$ harmonic
forms over the reduced manifold $[0,\epsilon_0) \times \bT^2$.   Now decompose this first summand further into functions
which are harmonic on $\bT^2$ and their orthogonal complement, with corresponding projection:
\begin{equation*}
L^{2}H^{*}_{\bT^2} (\calU; \calH_\theta) :=\{f \in L^2(M): \Delta_{\bT^{2}}f=0\}, \quad \Pi_y: L^{2}(M) \rightarrow L^{2}H^*_{\bT_2}(\calU; \calH_\theta).
\end{equation*}
As before, we write 
\[
\Pi_\theta \LapTY \Pi_\theta = \Pi_y \Pi_\theta \LapTY \Pi_\theta \Pi_y + \Pi_y^\perp \Pi_\theta \LapTY \Pi_\theta \Pi_y^\perp + E_2.
\]
As before, the extra cross-term $E_2$ decays at some rate as $x \to 0$.  This first summand on the right acts on a bundle
$\calH_{y,\theta}$ over $[0,\epsilon)$.

In terms of all of this, $\LapTY=\Delta_{\perp*}+\Delta_{0\perp}+\Delta_{00} + E$ where $E = E_1 + E_2$ and 
\begin{equation}\label{e:proj}
\Delta_{\perp*}=\Pi_{\theta}^{\perp} \LapTY \Pi_{\theta}^{\perp}, \ \Delta_{0\perp}=\Pi_{y}^{\perp}\Pi_{\theta} \LapTY \Pi_{\theta}\Pi_{y}^{\perp},
\ \Delta_{00}=\Pi_{y}\Pi_{\theta}\LapTY \Pi_{\theta} \Pi_{y}.
\end{equation}
We may regard each of these operators as acting on the entire neighborhood $\calU$, or alternately, on sections of the 
appropriate harmonic bundles over the reduced manifolds. 

The point of this decomposition is that these individual summands are each, up to multiplication by a power of $x$, fully elliptic in their 
respective calculi, i.e., in $\Diff_\bfa^2$, $\Diff_c^2$ and $\Diff_b^2$, respectively.     We now explain how to construct parametrices,
$G_{\perp *}$, $G_{0 \perp}$, $G_{00}$, for each of these principal summands.  Taking into account the various powers of $x$ weighting 
each component, we construct both left and right parametrices $G_{*,\ell}$ and $G_{*,r}$ for each of these components which include 
the appropriate `counterweights'.  Furthermore, the way in which the error term $E$ is handled is explained at the end.

% In the following, we also use the notation
% \[
% \calL_{\bfa} = \calL_{\perp *} + \calL_{\perp 0} + \calL_{00}
% \]
% for the corresponding components of $\calL_{\bfa} = x \LapTY$. 

\begin{remark}
For reasons that will become apparent below, the components of the corresponding decomposition of $\calL_\bfa$ will include certain 
additional powers of $x$ as prefactors; namely, we write
\begin{equation}
\calL_{\perp *} = x \Delta_{\perp *},\ \ \calL_{0 \perp } = x^{-1} \Delta_{0 \perp}, \ \ \mbox{and}\ \ \calL_{00} = x^{-3} \Delta_{00}.
\label{weightsL}
\end{equation}
\end{remark}

\bigskip

\noindent {\bf Parametrix for $\Delta_{\perp*}$}

The $\bfa$-symbol of the entire operator $\calL_{\bfa}$ is elliptic. This captures the uniform interior ellipticity (in the sense of $\bfa$-operators).
The projected operator $\calL_{\perp *} $ has the stronger property of full ellipticity: its restriction to the front face $\ff_{\bfa}$ 
is invertible on any polynomially weighted $L^2$ space. We may thus invoke the parametrix construction for fully elliptic $\bfa$-differential 
operators, as in \cite{GH2}, to invert this piece. 

More carefully, recall the projective coordinates $(S, Y) \in \RR^3$ defined near the front face $\ff_{\bfa}$ of the $\bfa$-double space 
$M^2_{\bfa}$.  In these coordinates, 
\[
\calL_{\perp *}=[(S{\tilde x}^{2}+1)^{3}\partial_{S}]^{2}+(S{\tilde x}^{2}+1)^{2}(\partial_{Y_{1}}^{2}+\partial_{Y_{2}}^{2})+ \Pi_\theta^\perp\partial_{\theta}^{2}
\Pi_\theta^\perp
\]
restricts to $\ff_{\bfa}$ as
\[
\del_S^2 + \Delta_Y + \Pi_\theta^\perp \del_\theta^2 \Pi_\theta^\perp.
\]
As advertised, this is invertible on $L^2( \RR^3 \times \bS^1)$ because $\Pi_\theta^\perp \del_\theta^2 \Pi_\theta^\perp \leq -1$ as a self-adjoint operator.
In fact, it is invertible on $\rho^s L^2 (\RR^3 \times \bS^1)$ for any $s \in \RR$, where $\rho = (1 + S^2 + |Y|^2)^{1/2}$. Hence there exists a solution $u$ to  
\[
(\del_S^2 + \Delta_Y + \Pi_\theta^\perp \del_\theta^2 \Pi_\theta^\perp) u(S,Y,\theta) = f(S,Y,\theta) \in \calC^\infty_0(\ff_{\bfa})
\]
which decays rapidly (and in fact, exponentially) as $\rho \to \infty$. This proves the full ellipticity of $\calL_{\perp *}$.   

Following \cite{GH2} (and in line with the typical elliptic parametrix construction in geometric microlocal analysis), the parametrix
for this component proceeds by first choosing an operator $\tilde{G}^{(1)}_{\perp *,r}$ which inverts the $\bfa$-symbol of $\calL_{\bfa}$ when
applied on the right, and is supported in a small neighborhood of the diagonal.  However, we must account for the restriction to the range of $\Pi_\theta^\perp$,
which can be done by using the partial symbol of $x\Delta$ in the $(x,y)$ (or $(S,Y)$) directions coupled with the globally
invertible operator $\Pi_\theta^\perp \del_\theta^2 \Pi_\theta^\perp$. The support of this inverse can be localized to a small
neighborhood of $\{ S=1, Y=0\} \times \bS^1 \times \bS^1$. 

This initial parametrix satisfies $\calL_{\perp *} \tilde{G}^{(1)}_{\perp *, r} = \mathrm{I} + \tilde{R}^{(1)}_r$, where 
the error term $\tilde{R}^{(1)}_r$ is smooth on $M^2_{\bfa}$ and supported near this same `extended' diagonal. 
The restriction of this error term to $\ff_{\bfa}$ is smooth and compactly supported.  We then correct this initial parametrix by
adding a second term, $\tilde{G}^{(2)}_{\perp *,r} \in \Psi^{-\infty}_{\bfa}$, the Schwartz kernel of which vanishes to all orders 
at all boundary faces except $\ff_{\bfa}$. Its restriction to this innermost front face satisfies 
\[
\calL_{\perp *} \tilde{G}^{(2)}_{\perp *,r} = -\tilde{R}^{(1)}_r + \tilde{R}^{(2)}_r \Rightarrow \calL_{\perp * } (\tilde{G}^{(1)}_r + \tilde{G}^{(2)}_r) = 
\mathrm{I} + \tilde{R}^{(2)}_r, 
\]
where the new remainder $\tilde{R}^{(2)}_r \in \Psi^{-\infty,1}_{\bfa}$, i.e., vanishes to first order on $\ff_{\bfa}$.  
Using the composition laws for the $\bfa$-calculus from \cite{GH2} and taking a Borel sum $\mathrm{I} + \tilde{S}$ of the Neumann series 
for $(\mathrm{I} + \tilde{R}^{(2)}_r)^{-1}$, we obtain a final parametrix 
\[
\tilde{G}_{\perp *, r} := (\tilde{G}^{(1)}_r + \tilde{G}^{(2)}_r) (\mathrm{I} + \tilde{S})
\]
which satisfies $\calL_{\perp *} \tilde{G}_r = \mathrm{Id} + \tilde{R}_r$, where $\tilde{R}_r \in \Psi^{-\infty, \infty}_{\bfa}$. In other 
words, this final error term vanishes to infinite order at all boundary faces of $M^2_{\bfa}$. 

Conjugating the equation $\calL_{\perp *} \tilde{G}_{\perp *, r} = \mathrm{I} + \tilde{R}_r$ by $x^{-1}$ gives
\[
x^{-1} \calL_{\perp *}  \tilde{G}_{\perp *, r} \, \tilde{x} = \mathrm{I} + x^{-1} \tilde{R}_r \tilde{x},
\]
so finally, setting $G_{\perp *, r} = \tilde{G}_{\perp *,r} \, \tilde{x}$, $R_r = x^{-1} \tilde{R}_r \tilde{x}$,  we conclude that
\[
\Delta_{\perp *} G_{\perp *,r} = \mathrm{I} + R_r.  % \ \ \mbox{where}\ \ G_{\perp *, r} = \tilde{G}_{\perp *,r} \tilde{x}, \ R_r = x^{-1} \tilde{R}_r \tilde{x},
\]
since $x^{-1}\calL_{\perp *} = \Delta_{\perp *}$ and $x^{-1}\, \mathrm{I} \,\tilde{x} = \mathrm{I}$. Here $\tilde{x}$ is the boundary defining function 
for $\del M$ lifted to the right factor of $M \times M$ 
and then to $M^2_{\bfa}$.    The conjugated error term has Schwartz kernel which still vanishes to infinite order at all boundary faces, hence lies in 
$\dot{\calC}^\infty(M\times M)$.  The new factor of $\tilde{x}$ in the parametrix itself  means that the Schwartz kernel of $G_{\perp *, r}$ vanishes to order $1$ 
on $\ff_{\bfa}$, which we write as $G_{\perp *,r} \in \Psi^{-2,1}_{\bfa}$.  This factor does not alter the infinite order vanishing at all other boundary faces. 

Using the analogous steps we can also construct a left parametrix $G_{\perp *, \ell} = \tilde{G}_{\perp *, \ell} \, \tilde{x} \in \Psi^{-2,1}_{\bfa}$.
% where $\tilde{G}_{\perp *, \ell}$ is a left parametrix for $\calL_{\perp *}$. 

As proved in \cite{GH2}, the left and right parametrices $\tilde{G}_{\perp *, \ell/r}$ have the expected boundedness properties on $\bfa$-Sobolev
spaces, and {\it relative to the volume form $dV_{\bfa} = dV_{g_{\bfa}}$}.  Note that the volume form for a Tian-Yau metric
satisfies $dV_{g_{\TY}} = x dV_{\bfa}$, so there is an equivalence
\[
x^c L^2(M, dV_{g_{\TY}}) = x^{c-1/2} L^2(M, dV_{\bfa}).
\]
The remainder terms map into $\dot{\calC}^\infty(M)$, hence are obviously compact on any weighted $\bfa$-Sobolev space. Incorporating the extra 
prefactor of $\tilde{x}$, this proves the 
\begin{proposition}
The operator $\Delta_{\perp *}$ admits left and right parametrices $G_{\perp*, \ell/r}\in \Psi_{\bfa}^{-2,1}(M)$ such that 
\[
\begin{aligned}
& G_{\perp*,\ell}\Delta_{\perp*} = \mathrm{I} + R_{\perp*,\ell}, \ \Delta_{\perp*}G_{\perp*,r} = \mathrm{I} + R_{\perp*,r}, \\
& \qquad \qquad R_{\perp *, \ell/r}  \in 
%\Pi_{\theta}^{\perp} x^{\infty}\Psi^{-\infty}_{\bfa}(M) \Pi_\theta^\perp = 
\Pi_\theta^\perp \, \dot{\calC}^\infty( M^2) \, \Pi_\theta^\perp.
\end{aligned}
\]
These yield bounded maps
\[
G_{\perp *, \ell/r} :  \Pi_\theta^\perp x^{c-1} H_{\bfa}^{\sigma}(M, dV_{\TY}) \rightarrow \Pi_\theta^\perp x^{c} H^{\sigma + 2}_{\bfa}(M, dV_{\TY}) \Pi_\theta^\perp
\]
and hence, by the compactness of the remainder terms, 
\[
\Delta_{\perp*}: \Pi_{\theta}^{\perp} x^{c}H_{\XX}^{\sigma+2}(M, dV_{\TY})\rightarrow\Pi_{\theta}^{\perp} x^{c-1}H_{\XX}^{\sigma}(M, dV_{\TY})
\] 
is Fredholm,  for any $c, \sigma \in \RR$.
\end{proposition}

\bigskip

\noindent {\bf Parametrix of $\Delta_{0\perp}$}

We now turn to the second summand 
\[
\Delta_{0\perp}=x^{-1} \left(  (x^{3}\partial_{x})^{2}+x^{2}\Pi_y^\perp \Pi_\theta \Delta_{\bT^{2}} \Pi_\theta \Pi_y^\perp \right),
\]
which becomes a cusp operator once we divide by an additional power of $x$:
\[
\calL_{0 \perp} = x^{-1} \Delta_{0\perp} = (x^2\del_x)^2+ x^3\del_x + \Pi_y^\perp \Pi_\theta \Delta_{\bT^{2}} \Pi_\theta \Pi_y^\perp,
\]
This is an elliptic combination of the basic cusp vector fields $x^2\del_x$ and $\del_{y_j}$, acting on sections of the bundle $\calH_\theta$ over 
$[0,\epsilon_0) \times \bT^2$, hence has invertible $c$-symbol. The term $x^3\del_x$ is lower order in the cusp algebra introduces
only perturbative correction terms. Just as in the previous case, $\calL_{0 \perp}$ is fully elliptic in the $c$-calculus. 

Regarding this operator as acting on functions in the range of $\Pi_y^\perp \Pi_\theta$ on the collar neighborhood $\calU$ of $\del M$,
we can construct a parametrix in the same three stages as above. Namely, the first term, obtained by the $c$-symbol calculus, is supported 
near the lift to $M^2_c$ of the submanifold $\{x = \tilde{x}\} \times (\bT^2)^2 \times (\bS^1)^2$.  This leaves a remainder which is smooth 
on $M^2_c$ and vanishes to infinite order at all boundary faces except the front face $\ff_c$.  The second part of the parametrix
is obtained by inverting the restriction of $\calL_{0 \perp}$ to this front face. In the projective coordinates 
$(s', y, \theta, \tilde{y},\tilde{\theta}, \tilde{x})$ on $M^2_c$, this restriction equals $\del_{s'}^2 + \Pi_y^\perp \Pi_\theta \Delta_{\bT^2} 
\Pi_\theta \Pi_y^\perp$, which is certainly invertible. In fact, if $f \in \Pi_y^\perp \Pi_\theta \calC^\infty_0(\RR_{s'}  \times \bT^2 \times \bS^1)$,
then there is a unique solution $u$ to $(\del_{s'}^2 + \Pi_y^\perp \Pi_\theta \Delta_{\bT^2} \Pi_\theta \Pi_y^\perp) u = f$ which decays rapidly as 
$|s'| \to \infty$. Hence this component is fully elliptic in the $c$-calculus.  The remainder term at this
stage is smooth, vanishes to order $1$ at $\ff_c$ and to infinite order at all other boundary faces.  The Neumann series 
argument leads to a final correction to a right parametrix $\tilde{G}_{0 \perp, r}$ which leaves a remainder term vanishing to all orders at all boundary faces. 
An identical argument gives a left parametrix $\tilde{G}_{0 \perp, \ell}$. 

Conjugating 
\[
\calL_{0 \perp} \tilde{G}_{0 \perp,r} = \mathrm{I} + \tilde{R}_{0,\perp,r}\ \mbox{and}\ \ \tilde{G}_{0 \perp, \ell} \, \calL_{0 \perp, \ell} = 
\mathrm{I} + \tilde{R}_{0\perp, \ell}
\]
by $x$ shows that $G_{0 \perp,\ell/r} = \tilde{G}_{0 \perp, \ell/r} \, \tilde{x}^{-1} \in \Psi_c^{-2,-1}$ are left and right parametrices for $\Delta_{0 \perp}$
with remainder terms vanishing to all orders at the boundaries of $M \times M$.  These act on functions in the range of $\Pi_y^\perp \Pi_\theta$.  
\begin{proposition}
The parametrices $G_{0\perp, \ell/r} \in \Psi_c^{-2, -1}(M)$ satisfy
\[
G_{0\perp,\ell} \Delta_{0\perp} = \mathrm{I} + R_{0\perp, \ell}, \ \Delta_{0\perp}G_{0\perp,r} = \mathrm{I} + R_{0\perp, r}
\]
where $R_{0\perp, \ell/r} \in \dot{\calC}^\infty(M^2)$. These operators are bounded as maps
\[
G_{0 \perp, \ell/r}: \Pi_y^\perp \Pi_\theta x^{c+1} H^\sigma_c(M, dV_{\TY}) \rightarrow  \Pi_y^\perp \Pi_\theta x^{c} H^{\sigma+2}_c(M, dV_{\TY})
\]
for any $c, \sigma \in \RR$, and by the obvious compactness of the remainder terms, 
\[
\Delta_{0\perp}:  \Pi_y^\perp \Pi_{\theta} x^cH_{c}^{\sigma+2}(M, dV_{\TY})\rightarrow \Pi_{y}^{\perp} \Pi_\theta x^{c+1}H_{c}^{\sigma}(M, dV_{\TY})
\]
is Fredholm for all such $c, \sigma$. 
\end{proposition}

\bigskip

\noindent {\bf Parametrix of $\Delta_{00}$}

The final term is even simpler.  Acting on the bundle $\calH_{y,\theta}$ of functions which are harmonic on both $\bS^1$ and $\bT^2$ over $[0,\epsilon_0)$, 
$\calL_{00} = x^{-3} \Delta_{00}$ reduces to the $1$-dimensional elliptic $b$-operator: 
\[
x^{-3} (x^{-1} ( (x^3\del_x)^2 - x^5\del_x)) = x^2 \del_x^2 + 2x \del_x.
\]
Using the projective coordinate $s = x/\tilde{x}$, this lifts to $[0,\epsilon_0)^2_b$ (or $M^2_b$) as 
\[
\calL_{00} = s^2\del_s^2 + 2s\del_s,
\]
and since the defining function $\tilde{x}$ for the front face $\ff_b$ in this coordinate system does not appear, full ellipticity corresponds
to the invertibility of this operator acting between weighted $b$-Sobolev spaces.  

Unlike for the previous two components, full ellipticity is not true for arbitrary values of the weight parameter.  Indeed, this dilation-invariant
operator on $\RR^+$ has a basis of homogeneous solutions consisting of the two functions $s^\gamma$, $\gamma_1 = -1, \ \gamma_2 = 0$.  These two
values of $\gamma$ are called the indicial roots of $\calL_{00}$. The monomial $s^\gamma$ just `misses' lying in $s^c L^2(\RR^+, ds/s^3)$ 
when $c = \gamma-1$ (in the sense that $s^\gamma$ lies in $s^{c-\epsilon}L^2$ but is excluded from $s^{c+\epsilon} L^2$ near $s=0$
for any $\epsilon > 0$, with the analogous statement near $s=\infty$), so we define the indicial weights $c_j = \gamma_j - 1$, i.e., $c_1 = -2$,
$c_2 = -1$.   The basic fact then is that $\calL_{00}: s^c H^{\sigma+2}_b( \RR^+, s^{-3}ds) \to s^c H^\sigma_b(\RR^+, s^{-3}ds)$ is
invertible if and only if $c \not\in \{c_1, c_2\}$.  When $c$ is equal to one of these two indicial values, this map does not have closed range. 
We remark that we are using the measure $s^{-3}ds$ here since it corresponds to the Tian-Yau 
measure $dV_{\TY} \cong x^{-3} dx dy d\theta$. This key fact can be proved easily using the Mellin transform.  In summary,
$\calL_{00}$ is fully elliptic on these weighted $b$-Sobolev spaces for all weight parameters $c \neq -2, -1$. 

We proceed as before to construct parametrices for $\calL_{00}$ in the $b$-calculus. This is particularly elementary in this 
$1$-dimensional setting, see \cite{APS, Ma-edge1}.  We first choose a symbolic inverse supported near $\{x = \tilde{x}\} 
\times (\bT^2)^2 \times (\bS^1)^2$, then a correction term which solves away the leading coefficient of the initial
remainder term on the front face, and finally, using a Neumann series, a third term which leads to a remainder term
vanishing to infinite order on the front face.  All of these operators act on  functions which are harmonic on both $\bT^2$ and $\bS^1$. 
%parametrix $\tilde{G}_{00} \in \Psi^{-2}_b(M)$
The important new feature is that in the second step of the construction, where we solve away the restriction of the remainder term on 
the front face $\ff_b$,  we must choose the solution $u$ to $ (s^2\del_s^2  + 2s \del_s) u = f \in \calC^\infty_0(\RR^+)$ differently,
according to whether $c < -2$, $-2 < c < -1$ or $c > -1$.  Indeed, if $f$ is supported in $[a,b]$, and if $c < -2$, then we choose the 
solution $u$ which is a linear combination of $s^{-1}$ and $s^0$ for $0 < s \leq a$ and which vanishes for $s > b$; if $-2 < c < -1$, 
we choose the solution $u$ which is a multiple of $s^{0}$ for $s \leq a$ and of $s^{-1}$ for $s \geq b$; finally, if $c > -1$, we choose 
$u$ to vanish for $s \leq a$ and to be a linear combination of $s^0$ and $s^{-1}$ when $s \geq b$.  In each of these cases, the resulting
correction term to the parametrix has support spreading to either the left face of $M^2_b$, where $x=0$, or the right face, where $\tilde{x} = 0$,
or both.  The final parametrices $\tilde{G}_{00, \ell/r}$ are polyhomogeneous at the left and right faces:
\[
\tilde{G}_{00, \ell/r} \in \begin{cases}  \Psi_b^{-2, 0, -1, \emptyset},\ \ & c < -2 \\ \Psi_b^{-2, 0, 0, 1},\ \ &-2 < c < -1 \\
\Psi_b^{-2, 0, \emptyset, 0},\ \ &c > -1.  
\end{cases}
\]
The parametrix for $\Delta_{00}$ is obtained by conjugating by $x^3$.  Thus
\begin{equation}
G_{00, \ell/r} = \tilde{G}_{00, \ell, r} \tilde{x}^{-3} \in 
\begin{cases}  \Psi_b^{-2, -3, -1, \emptyset},\ \ & c < -2 \\ \Psi_b^{-2, -3, 0, -2},\ \ &-2 < c < -1 \\
\Psi_b^{-2, -3, \emptyset, -3},\ \ &c > -1.  
\end{cases}
\label{diffG00}
\end{equation}

\begin{proposition}
For each $\sigma \in \RR$ and $c \in \RR \setminus \{-2, -1\}$ there exist parametrices $G_{00, \ell/r}$ as in \eqref{diffG00} such that 
\[
G_{00, \ell} \Delta_{00} = \mbox{I} + R_{00, \ell}, \ \ \Delta_{00}G_{00, r} = \mbox{I} + R_{00, r},
\]
where $R_{00, \ell, r}$ are polyhomogeneous on $M^2$ (without needing to blow up), with index sets on the left and right faces precisely
the same as those for $G_{00, \ell/r}$. Then 
\[
G_{00, \ell/r}:  x^{c+3} \Pi_y \Pi_\theta H^{\sigma}_b(M, dV_{\TY}) \rightarrow x^c \Pi_y \Pi_\theta H^{\sigma+2}_b(M, dV_{\TY})
\]
are bounded, and 
\[
\Delta_{00}: \Pi_{y}\Pi_{\theta}\, x^{c}H_{b}^{\sigma+2}(M, dV_{\TY}) \rightarrow \Pi_{y}\Pi_{\theta}\, x^{c+3}H_{b}^{\sigma}(M, dV_{\TY})
\]
is Fredholm for any $c \neq -2, -1$ and any $\sigma\in \RR$. 
\end{proposition}

\bigskip

\noindent {\bf Nontrivial fibrations}

Before the final step in this construction, we record the few minor changes needed when the $\bS^1$ fibration of $Y$ is nontrivial. 
In fact, the only real difference is that the operator $x \LapTY$ has a slightly more complicated appearance simply because 
the vector field $x\del_{y_2}$ must be replaced by $x(\del_{y_2} - y_1 \del_\theta)$.  The extra term in this only appears in the 
component $\Delta_{\perp *}$. 

As before, $x\LapTY$ is still fully elliptic. Indeed, the sequence of blowups to obtain $M^2_{\bfa}$ is still the same, and 
using the `iterated' projective coordinates $(S, Y, \theta)$, the lift of $x(\del_{y_2} - y_1\del_\theta)$ to this final blowup 
equals 
\[
(1 + \tilde{x}S)( \del_{Y_2} - \tilde{x}(\tilde{y} + \tilde{x}Y_1) \del_\theta),
\]
which has the same restriction $\del_{Y_2}$ on $\ff_{\bfa}$ as in the product case. 

The decomposition of $x\LapTY$ into its various components, and the parametrix construction for each of these components,
goes exactly as before.

\bigskip

\noindent {\bf Assembling the components}

At last, we add these three components of the left and right parametrices to obtain the operators
\[
\begin{aligned}
& G_{\ell/r} :=  G_{\perp *,\ell/r} + G_{0 \perp, \ell/r} + G_{00,\ell,r}, \\ 
G_{\perp *,\ell/r} \in &  \Psi^{-2,1}_{\bfa}(M),\  G_{0 \perp,\ell/r} \in \Psi^{-2,-1}_c(M),\ G_{00,\ell/r} \in \Psi^{-2, -3, E_{10}, E_{01}}_b(M),
\end{aligned}
\]
where $E_{10}$ and $E_{01}$ are the index sets at the left and right faces of $M^2_b$ chosen according to the weight $c$.  

If the fibration is extended to some fixed order into the interior, we must still incorporate the cross-term $E$, which accounts for the extent to
which the fiber-harmonic parts interact with their orthogonal complements.  This term is lower order at $\del M$, hence does not affect
the steps in the parametrix construction where the model problems on the various front faces must be inverted. Hence the inclusion
of this term produces only higher order terms in the expansions at these front faces, and at the side faces for $\Delta_{00}$.  A further
Neumann series iteration argument produces final correction terms to $G_{\ell/r}$, which result in remainder terms which are polyhomogeneous 
on $M^2$.

\begin{theorem}\label{thm:main}
If $g_{TY}$ is any metric with Tian-Yau asymptotics on $M^4$, then for any $c \in \RR \setminus \{-2, -1\}$ and $\sigma \in \RR$,
\begin{multline}
\Delta: x^{c}H_{\XX}^{\sigma+2}(M, dV_{\TY}) \longrightarrow  \Pi_{\theta}^{\perp} x^{c-1}H_{\XX}^{\sigma}(M, dV_{\TY})\oplus \\
\Pi_{\theta}\Pi_{y}^{\perp} x^{c+1} H_{\XX}^{\sigma}(M, dV_{\TY}) \oplus  \Pi_{y}\Pi_{\theta}x^{c+3}H_{\XX}^{\sigma}(M, dV_{TY})
\end{multline}
is Fredholm. 
\end{theorem}

\begin{corollary}
If $u \in x^c H^{\sigma+2}_{\bfa}(M, dV_{\TY})$ satisfies $\LapTY u = 0$, then $u$ is polyhomogeneous, with index set $E$ containing 
only pairs $(\gamma, k)$ with $\gamma > c +1$.  
\label{phgcor}
\end{corollary}
\begin{proof} Applying the left parametrix $G_\ell$ to the equation $\LapTY u = 0$ gives that $u =- R_\ell u \in \calA_{\phg}$.
The exponents that can arise in the index sets are those stemming from indicial roots which lie in $x^c L^2$.  
\end{proof}

\section{$L^{2}$ cohomology}
As a first application of these analytic tools, we determine the space of $L^2$ harmonic forms for any ALH* gravitational instanton.  

The space $L^2\calH^k(M,g)$ of $L^{2}$ harmonic forms on any complete manifold $(M,g)$ is isomorphic to the {\it reduced} $L^2$ cohomology
\[
\overline{L^2H}^{\, k}(M,g) = \{\omega \in L^2\Omega^k(M): d\omega = 0\} \big/ \, \overline{\{ d\eta: \eta \in \calC^\infty_0\Omega^{k-1}(M)\}},
\]
and thus depends only on the quasi-isometry class of the metric $g$.  This implies that we may assume that 
\[
g_{\TY} = g_{\GH}
\]
in a neighborhood of $\del M$. This simplifies various computations below, though the various decay estimates, etc., are all valid
for a Tian-Yau metric which is only asymptotic to this model. 

The key difficulty is that if the range of $d$ on $L^2\Omega^{k-1}$
is not closed, then the usual tools of localization, e.g.\ Mayer-Vietoris and variants of the Poincar\'e lemma, are not directly applicable. 
The alternate strategy is one employed by Zucker and later in settings close to the present one in \cite{HHM}. It  
involves first establishing an isomorphism of the space of $L^2$ harmonic forms with a certain {\it weighted} cohomology space
where the range of $d$ is closed, and then localizing using the sheaf-theoretic characterization of intersection cohomology to equate this
weighted cohomology with an intersection cohomology space for a particular perversity $\mathfrak p$ which we finally identify with
a more familiar cohomology space. 

\subsection{The Hodge-de Rham operator}
We first calculate the expression for the operator $D = d + \delta$ relative to a particular coframe near the boundary.

First define the orthonormal (with respect to $g_{\GH}$) frame field
\[
X_{1}=x^{5/2} \partial_{x}, \ X_{2}=x^{1/2}\partial_{y_{1}}, \ X_{3}=x^{1/2} (\partial_{y_{2}}-y_{1}\partial_{\theta}), \ X_{4}=x^{-1/2}\partial_{\theta};
\]
and dual coframe field
\begin{equation}
e_{1}=x^{-5/2}dx, \ e_{2}=x^{-1/2}d{y_{1}}, \ e_{3}= x^{-1/2}dy_{2}, \ e_{4}=x^{1/2}(d\theta+y_{1}dy_{2}).
\label{1forms}
\end{equation}
We then decompose any form of even degree as 
\begin{equation}\label{e:evenform}
f_{0} + \ f_{ij} e_{i} \wedge e_{j} + \ f_{1234} e_{1}\wedge e_{2}\wedge e_{3} \wedge e_{4},
\end{equation}
while a form of odd degree takes the form
\begin{equation}\label{e:oddform}
f_{i} e_{i} + \ f_{ijk} e_{i }\wedge e_{j} \wedge e_{k}.
\end{equation}
Any $k$-form decomposes as
\begin{equation}\label{e:formsplit}
\omega = \frac{dx}{x^{5/2}} \wedge \frac{\alpha}{x^{(k-2)/2}}\wedge e_{4} + \frac{dx}{x^{5/2}} \wedge \frac{\tilde \alpha}{x^{(k-1)/2}} 
+  \frac{\beta}{x^{(k-1)/2}} \wedge e_{4}+ \frac{\tilde \beta}{x^{k/2}},
\end{equation}
where $\alpha$, $\tilde{\alpha}$, $\beta$ and $\tilde{\beta}$ involve only the forms $e_2$ and $e_3$. 
We write $D$ as a matrix acting on these coefficients $(f_*)$ from~\eqref{e:evenform} and~\eqref{e:oddform}.

The exterior differential $d$ splits as %into three terms %differentiation with respect to each of the three distinguished directions
\[
d=d_{x} + d_{B} + d_{F},
\]
where, by slight abuse of notation, $d_{x}=e_{1}\wedge\partial_{X_{1}}$, $d_{F}=e_{4}\wedge \partial_{X_{4}}$, and 
$d_{B}=e_{2}\wedge \partial_{X_{2}} + e_{3}\wedge\partial_{X_{3}}$ is the differential in the horizontal subspace.
There is a corresponding splitting of the codifferential $\delta$ as
\[
\delta=\delta_{x} + \delta_{B} + \delta_{F}.
\]

As observed by Bismut and Cheeger \cite{BismutCheeger}, the fact that the Levi-Civita connection for the family of metrics $g_x = 
g|_{Y \times \{x\}}$ does not preserve the horizontal and vertical subspaces of $TY$ complicates some computations, so it is
more convenient to introduce a new connection which does preserve these subspaces. We review this briefly, referring to
\cite{HHM} for more details. We follow the notation from that paper.

The idea is to add some new terms to the Levi-Civita connection which involve the second fundamental forms of the fibers and the
curvature of the horizontal distribution.  %We compute these quantities now to keep track of the $x$ rescaling.
These vertical and horizontal distributions are spanned by $U=x^{-1/2}\partial_{\theta}$ and $\tilde X_{2}=x^{1/2}\partial_{y_{1}},\tilde X_{3}=
x^{1/2}(\partial_{y_{2}}-y_{1}\partial_{\theta})$, respectively.  By definition, the second fundamental form of the fibers is 
\[
\mathbb{II}_{\tilde X_{i}}(U,U)=\langle \nabla^{L}_{U}U, \tilde X_{i} \rangle, \ i=2,3,
\]
but since $\nabla^{L}_{U}U=x^{-1}\nabla_{\partial_{\theta}} {\partial_{\theta}}=-\frac{x^{4}}{2}\partial_{x}$, we have that
\[
\mathbb{II}_{\tilde X_{2}}(U,U) = \mathbb{II}_{\tilde X_{3}}(U,U) = 0.
\]
Next, the curvature of the horizontal distribution is 
$$
\mathcal{R}(\tilde X_{i}, \tilde X_{j})=P^{V}([\tilde X_{i}, \tilde X_{j}]),
$$
where $P^V$ is the projection onto the vertical tangent bundle and off the horizontal distribution. The only commutator with nontrivial vertical component is 
\[
[\tilde X_{2}, \tilde X_{3}]=-x\partial_{\theta},
\]
whence 
\begin{equation}
\mathcal{R}(\tilde X_{2}, \tilde X_{3})=-\mathcal{R}(\tilde X_{3}, \tilde X_{2})=-x^{3/2}U; \ \mathcal{R}(\tilde X_{i}, \tilde X_{j})=0 \text{ all other } i, j.
\label{curvR}
\end{equation}
% Finally, 
% \[
% \mathbb{II} (e_{4})=0
% \]
% and 
% \[
% \mathcal{R}(e_{4})=-2x^{3/2} \frac{dy_{1}}{x^{1/2}} \wedge \frac{ dy_{2}}{x^{1/2}}=-2 x^{3/2} e_{2}\wedge e_{3}.
% \]
We can thus rewrite
$$
d_{Y}=d_{F} + \tilde d_{B} - \frac12 \mathcal{R}
$$
and 
\[
\delta_Y = \delta_F + \tilde{\delta}_B - \frac12 \mathcal{R}^*,
\]
where $\tilde{d}_B$ denotes the lift of $d_{\bT^2}$ to the horizontal distribution and $\tilde{\delta}_B$ is its adjoint.  To simplify
notation, we set $R = -\frac12 \mathcal R$ below.

According to the splitting in \eqref{e:formsplit}, the action of $D = D_{\GH} =d+\delta$ on $\Omega^k$ takes the form
\begin{equation}\label{e:Matrix2}
\begin{pmatrix}
x^{1/2}\tilde D_{B} & x^{-1/2}d_{F}+  x^{3/2} R^{*} & -\frac{k-2}{2} x^{3/2} + x^{5/2}\partial_{x} & 0 \\
x^{-1/2} \delta_{F} + x^{3/2} R&  x^{1/2} \tilde D_{B} & 0 & -\frac{k}{2} x^{3/2} + x^{5/2} \partial_{x} \\
-\frac{4-k}{2} x^{3/2} + x^{5/2}\partial_{x} &0&  x^{1/2} \tilde D_{B} & x^{-1/2}d_{F}+ x^{3/2} R^{*}\\
0 & -\frac{2-k}{2} x^{3/2}+ x^{5/2}\partial_{x} & x^{-1/2}\delta_{F} + x^{3/2} R & x^{1/2} \tilde D_{B}
\end{pmatrix}
\end{equation}
Observe that $x^{1/2} D_{\GH}$ is an elliptic $\bfa$-operator.  It decomposes into components acting on the ranges of the various projectors
$\Pi_\theta$, $\Pi_\theta^\perp$, $\Pi_y$ and $\Pi_y^\perp$.  Exactly as for the scalar operator detailed in the last section, an inspection of \eqref{e:Matrix2}
leads to the following conclusions:
\begin{proposition}
The component $\calL_{\perp *}$ of $x^{1/2}D_{\GH}$ on the range of $\Pi_\theta^\perp$ is a fully elliptic $\bfa$-operator.  Similarly,
the component $\calL_{0 \perp}$ of $x^{-1/2} D_{\GH}$ on the range of $\Pi_y^\perp \Pi_\theta$ is a fully elliptic $c$-operator.  The 
corresponding parametrices satisfy $G_{\perp *, \ell/r} = \tilde{G}_{\perp *, \ell/r} \tilde{x}^{1/2} \in \Psi_{\bfa}^{-1, 1/2}$ and
$G_{0 \perp, \ell/r} = \tilde{G}_{0 \perp, \ell/r} \tilde{x}^{-1/2} \in \Psi_c^{-1, -1/2}$, and their remainder terms are rapidly decaying at
all boundaries of $M^2$.
\end{proposition}
\begin{corollary}
If $\omega \in L^2 \calH^k(M, g_{\GH})$, then the components $\Pi_\theta^\perp \omega$ and $\Pi_y^\perp \Pi_\theta \omega$ decay
rapidly at $\del M$. 
\end{corollary}

Finally, the component $D_{00}:=\Pi_{\theta}\Pi_{y}D\Pi_{\theta}\Pi_{y}$ satisfies
\[
x^{-3/2}D_{00}= \calL_{00} = 
\begin{pmatrix}
0 &R^{*} & -\frac{k-2}{2}  + x\partial_{x} & 0 \\
 R& 0& 0 & -\frac{k}{2} + x  \partial_{x} \\
-\frac{b-k+2}{2}  + x\partial_{x} &0&  0 &R^{*}\\
0 & -\frac{b-k}{2}+ x \partial_{x} &R &0
\end{pmatrix}.
\]
Note  that $R$ and $R^*$ contribute terms of the same order as the principal part, and hence must be taken into account in the
computation of indicial roots. Inserting the expressions from \eqref{curvR}, then using the notation~\eqref{e:evenform} and~\eqref{e:oddform} 
for $\omega = \omega_{even} + \omega_{odd}$, 
\begin{equation}\label{eq:D00}
\begin{aligned}
x^{-3/2} & D_{00} (\omega_{\mathrm{even}})=x\del_{x}f_{0} \, e_{1}  - x\del_{x} f_{12} \, e_{2}  -x\del_{x} f_{13} \, e_{3} \\
&+ \left(-(x\del_x-1) f_{14} + f_{23}\right)  e_{4} +\left((x\del_x-1) f_{23}- f_{14}\right) e_{1}\wedge e_{2} \wedge e_{3}\\
&+ x\del_{x} f_{24} \, e_{1}\wedge e_{2}\wedge e_{4} + x\del_{x}  f_{34} \, e_{1}\wedge e_{3}\wedge e_{4}\\ 
&- x\del_{x}  f_{1234} \, e_{2}\wedge e_{3}\wedge e_{4}.
\end{aligned}
\end{equation}
This mostly decouples, so it is straightforward to see that this has indicial root $0$ for $f_{0}, f_{12}, f_{13}, f_{24}, f_{34}, f_{1234}$, and indicial roots
$0$ and $2$ for $f_{14}, f_{23}$. In particular they satisfy $f_{14}=-f_{23}$ for indicial root 0 and $f_{14}=f_{23}$ for indicial root 2.  Similarly, 
\begin{equation}
\begin{aligned}
x^{-3/2}&D_{00}(\omega_{\mathrm{odd}})=-(x\del_{x}-\tfrac12)f_{1} \\
&  +(x\del_{x}-\tfrac12 )f_{2} \, e_{1}\wedge e_{2} +(x\del_x-\tfrac12 )f_{3} \, e_{1}\wedge e_{3}\\
&+ \left( (x\del_x+\tfrac12 ) f_{4} - f_{123}\right) \, e_{1}\wedge e_{4} +  \left( -(x\del_x+\tfrac12 )f_{123} + f_{4}\right) \, e_{2}\wedge e_{3}\\
&  -(x\del_x-\tfrac12 )f_{124} \, e_{2}\wedge e_{4} -(x\del_x-\tfrac12)f_{134} \, e_{3}\wedge e_{4}\\
& + (x\del_x-\tfrac12)f_{234} \, e_{1}\wedge e_{2}\wedge e_{3}\wedge e_{4},
\end{aligned}
\end{equation}
which has indicial root $1/2$ for $f_{1}, f_{2}, f_{3}, f_{124}, f_{134}, f_{234}$ and $-3/2$, $1/2$ for $f_{4}, f_{123}$. 
Altogether, $\{ -\frac{3}{2}, 0, \frac{1}{2}, 2\}$ constitutes the full set of indicial roots. 
%where $0$ and $1/2$ have multiplicity.

% This set of indicial roots for $x^{-3/2} D_{00}$ is symmetric with respect to $1/4$. This midpoint may seem unexpected, but can be explained
% as follows. The $b$-operator
% $$
% \tilde D_{00}=x^{-3/4}D_{00} x^{-3/4} 
% $$
% is self-adjoint with respect to the volume form $dV= x^{-3} dxdy_{1}dy_{2}d\theta$, and for this measure, the $L^2$ cutoff  (i.e., the exponent
% for which $x^\gamma$ just fails to lie in $L^2$) is $\gamma_c =1$.  Now, $\tilde \alpha$ is an indicial root for $\tilde D_{00}$ if 
% $\tilde D_{00}(x^{\tilde \alpha}) = (x^{-3/4}D_{00}) (x^{\tilde \alpha-3/4})=\calO(x^{\tilde \alpha+1})$, hence
% $(x^{-3/2} D_{00}) (x^{\tilde \alpha-3/4})=\calO(x^{\tilde \alpha+1/4})$. In other words, $\alpha:=\tilde \alpha-3/4$ is an indicial root 
% for $x^{-3/2}D_{00}$.   Altogether, the indicial root set for $\tilde D_{00}$ is $\{-3/4, 3/4, 5/4, 11/4\}$, which is indeed symmetric 
% around the $L^{2}$ cutoff $\gamma_c$.  

Constructing left and right parametrices $\tilde{G}_{00, \ell/r}$ for $x^{-3/2}D_{00}$ in the $b$-calculus, we obtain parametrices
$G_{00, \ell/r} = \tilde{G}_{00, \ell/r} \tilde{x}^{-3/2}$ for $D_{00}$ itself. Adding these to the parametrices for the other components
to obtain the full parametrices $G_{\ell/r}$, we obtain the main analytic result:
\begin{proposition}\label{prop:dirac}
Suppose that $c$ is a nonindicial weight, i.e., not equal to $\gamma_j - 1$ for any of the indicial roots for $x^{-3/2} D_{00}$ listed above.
Then %there exist left and right parametrices $G_{\ell/r}$ for $\Delta_{\GH}$ 
\begin{multline*}
G_{\ell/r} :  x^{c+3/2} \Pi_y \Pi_\theta H_{\bfa}^{\sigma}\Omega^{*}(M, dV_{\GH})  \oplus x^{c+1/2} \Pi_y^\perp \Pi_\theta H_c^{\sigma} \Omega^*(M, dV_{\GH}) \\
\oplus x^{c-1} \Pi_\theta^\perp H^{\sigma}_b \Omega^*(M, dV_{\GH}) \longrightarrow x^{c}H_{\bfa}^{\sigma+1}\Omega^{*}(M, dV_{\GH})). 
\end{multline*}
The remainders $R_{\ell/r}$ are polyhomogeneous on $M^2$, with index sets $E_{10}$ and $E_{01}$ at the left and right faces of this product space.
These index sets are determined by the component of $\RR \setminus \{\mathrm{indicial roots} \}$. 
%$$
%\Pi_{\theta}\Pi_{y}x^{\infty}\Psi_{\XX}^{-\infty}(M) \oplus \Psi^{-\infty, E_\ell, E_r}(M).
%$$
%The final space on this last line consists of smoothing operators on $M$ with Schwartz kernels which are polyhomogeneous on the simple
%product $M^2$ (without any blowups), for some index sets $E_\ell$ and $E_r$ at the two boundary faces $\del M \times M$ and $M \times \del M$.

If $\omega$ is a harmonic form in $x^c L^2\Omega^*(M, dV_{\GH})$, then $\omega$ is polyhomogeneous with exponents coming from
the indicial roots with $\gamma_j > c + 1$.  In particular, if $c = 0$, then only the harmonic forms with initial exponent $2$ arise.
\label{Darametrix}
\end{proposition}

There is a corresponding regularity result:
\begin{proposition}\label{prop:poly}
Let $\omega \in x^\alpha L^2 \Omega^*(M)$ for any $\alpha \in \RR$. 
\begin{itemize} 
\item If $D\omega=0$, then $\omega \in \calA_{\mathrm{phg}}^*\Omega^*(M)$, i.e., $\omega$ is polyhomogeneous;
\item if $D\zeta=\eta$ where $\eta\in \mathcal{A}^{a}\Omega^{*}(M)$ for 
\[
\zeta\in x^{c-2}\Pi_{\theta}\Pi_{y} H^{1} \Omega^{*}_{\XX}(M)  \oplus x^{c-1}\Pi_{\theta}\Pi_{y}^{\perp} H^{1} \Omega^{*}_{\XX}(M) \oplus 
x^{c}\Pi_{\theta}^{\perp}\Pi_{y}^{\perp} H^{1}\Omega^{*}_{\XX}(M),
\]
with $c < a$, and if neither $a$ nor $c$ are indicial roots, then 
\[
\zeta \in \Pi_{\theta}\Pi_{y} \mathcal{A}^{E}_{phg}\Omega^{*}_{\XX}(M) + \mathcal{A}^{a}\Omega^{*}(M),
\]
where $E$ is the set of indicial roots in the interval $(c,a)$. 
\end{itemize}
\end{proposition}
\begin{proof}
If $D\omega=0$, then applying the left parametrix gives $\omega = G_{\ell}D\omega - R_{\ell}\omega = - R_\ell \omega$.  Using the
structure of the remainder operators as explained in Proposition~\ref{prop:dirac}, we obtain that
$R_{\ell}\omega \in \calA_{\mathrm{phg}}^{E_\ell}(M)$.

Next, if $D\zeta=\eta \in \calA^{a}\Omega^{*}(M)$, then $\zeta =G_{\ell}\eta - R_{\ell}\zeta$. We invoke the fact that $G_\ell$ preserves conormality
and order of vanishing, so the first term on the right lies in $\calA^a \Omega^*$, while the second term is polyhomogeneous, with exponents all
strictly greater than $c$. 

We have argued here assuming that the cross-terms in the expression for $D$ above vanish to infinite order. If they do not, then an iteration
to successively improve the regularity of $\omega$ still leads to the polyhomogeneity of solutions to $D\omega = 0$. On the other hand,
the second result is valid only if $a-c$ is no larger than the order of vanishing of these cross-terms. 
\end{proof}

\subsection{Hodge theory for ALH* spaces}
We now return to a consideration of the Hodge cohomology space $L^2 \mathcal H^k(M, g) = \{\omega \in L^2(M,g): d\omega = \delta \omega = 0\}$,
using the strategy outlined in the preamble to this section. 

%As is well-known, this is the same as the reduced $L^2$ cohomology, which is the quotient of the space of closed $L^2$ forms by the {\it closure} of the 
%space $\{d\eta: \eta \in L^2 \Omega^{k-1}, d\eta \in L^2 \Omega^k\}$.  It is precisely the passage to this closure which makes the computation
%of this Hodge cohomology more difficult.  Indeed, there are a number of sheaf-theoretic and other topological tools which make the computation 
%of unreduced $L^2$ cohomology much more tractable, but if the range of $d$ is not closed, then the quotient is infinite dimensional (and
%non-Hausdorff). 
%As explained in the preamble of this section, we first identify this with a certain weighted $L^2$ cohomology space, chosen so that
%range of $d$ is closed, and then computing this by a sheaf-theoretic localization argument. 

%To get started on this, 
%note that the expression for $d$ with respect to the coframe $e_1, e_2, e_3, e_4$ is of the form $x^{3/2}$ times a 
%$b$-operator. \rafe{Is this true?}  Because of this, we define, for any degree $k$, the Hilbert complex
%$$
%\dots \overset{d}{\longrightarrow} x^{a-\frac{3}{2}}L^{2}\Omega^{k-1} \overset{d}{\longrightarrow} x^{a} L^{2}\Omega^{k} \overset{d}{\longrightarrow} 
%x^{a+\frac{3}{2}}L^{2}\Omega^{k+1} \overset{d}{\longrightarrow} \dots.
%$$
%This in turn, and again 
We begin, following \cite{HHM}, by defining the family of weighted de Rham cohomology spaces
\[
\WH^{k}(M,g,a) =\frac{\{\omega \in x^{a}L^{2}\Omega^{k}: d\omega=0\}}
{\{d\eta: \eta \in x^{a-\frac{3}{2}}L^{2}\Omega^{k-1}, d\eta \in x^{a}L^{2}\Omega^{k}\}}
\]

\begin{theorem}\label{Th:weighted}
For $\epsilon > 0$ sufficiently small, there is an isomorphism
$$
\Phi: L^{2}\mathcal{H}^{k}(M,g) \longrightarrow  \text{Im}\left(\WH^{k} (M, g, \epsilon) \rightarrow \WH^{k} (M,g, -\epsilon)\right).
$$
\end{theorem}
\begin{proof}	
We first show that $\Phi$ is well-defined.  Indeed, if $\omega \in L^2 \mathcal H^k$, then $D\omega = 0$, and we have seen that
as a consequence, $\omega$ is polyhomogeneous. The leading term in its expansion must be $x^2$; this is because $2$ is the only indicial root 
greater than the $L^2$ cutoff decay rate $1$. In fact, by~\eqref{eq:D00}, the leading term in this expansion equals
\begin{align}
 x^{2} f_{14} \left( \frac{dx}{x^{5/2}} \wedge  x^{1/2}(d\theta + y_{1}dy_{2})  + \frac{dy_{1}}{x^{1/2}} \wedge \frac{dy_{2}}{x^{1/2}}\right)
\end{align}
where $f_{14}$ is smooth for $x \geq 0$. Hence there is a well-defined class $[\omega] \in \WH(\epsilon)$ for any $\epsilon<1$. The image $\Phi(\omega)$ 
is the image of this class in $\WH(-\epsilon)$. 
 
We next prove that $\Phi$ is injective. For this, suppose that $\Phi([\omega])=0$. Then $\omega=d\zeta$ for some $\zeta \in x^{-\epsilon-3/2} L^2$. 
Write
$$
 \omega=\frac{dx}{x^{5/2}} \wedge \frac{\alpha}{x^{(k-2)/2}}\wedge e_{4} + \frac{dx}{x^{5/2}} \wedge \frac{\tilde \alpha}{x^{(k-1)/2}} 
+  \frac{\beta}{x^{(k-1)/2}} \wedge e_{4}+ \frac{\tilde \beta}{x^{k/2}} 
$$ 
where $\alpha, \tilde \alpha, \beta, \tilde \beta \in \calO(x^{2})$,  and
$$
\zeta=\frac{dx}{x^{5/2}} \wedge \frac{\mu}{x^{(k-3)/2}}\wedge e_{4} + \frac{dx}{x^{5/2}} \wedge \frac{\tilde \mu}{x^{(k-2)/2}} 
+  \frac{\nu}{x^{(k-2)/2}} \wedge e_{4}+ \frac{\tilde \nu}{x^{(k-1)/2}} 
$$ 
where $\mu, \tilde \mu, \nu, \tilde \nu \in \calO(x^{-\epsilon - 1/2})$ (note the $L^{2}$ cutoff is 1).
It suffices then to check that these orders of vanishing justify the equality
$$
\langle \omega, d\zeta\rangle = \langle \delta \omega, \zeta \rangle=0.
$$
To see this, note that the boundary term (say in the integral over $x \geq x_0 > 0$) in this integration by parts equals 
 $$
\int_{x=x_0} \zeta\wedge *\omega.
$$ 
Writing this out carefully, there are two tangential terms in $\zeta \wedge *\omega$, which contribute the terms
 \begin{align}
 \int \frac{\nu}{x^{\frac{k-2}{2}}} \wedge x^{\frac{1}{2}}(d\theta + y_{1}dy_{2})\wedge \frac{*_{Y}\alpha}{x^{\frac{2-(k-2)}{2}}},\\
 \int \frac{\tilde \nu}{x^{\frac{k-1}{2}}} \wedge \frac{*_{Y}\tilde \alpha}{x^{\frac{2-(k-1)}{2}}}\wedge x^{\frac{1}{2}}(d\theta + y_{1}dy_{2}).
 %\int \frac{\tilde \nu}{x^{\frac{k-1}{2}}} \wedge \frac{*_{Y}\alpha}{x^{\frac{2-(k-2)}{2}}}.
\end{align}
Both integrands decay at a sufficient rate with order $\calO(x^{1-\epsilon})$. Since $\epsilon<1$, both of these boundary terms vanish, which proves that $\phi$ is injective. 
%; the last one has the least order of vanishing, is of the form $\calO(x^{\epsilon})$. 

Finally we prove the surjectivity of $\Phi$. For this, recall the decomposition
\begin{equation}\label{eq:decomposition}
 x^{-\epsilon} L^{2}=\ran D_{-\epsilon} \oplus L^2 \calH^*
\end{equation}
where $D_{-\epsilon}$ is short for the operator $D$ acting on the space $\Pi_\theta \Pi_y x^{-\epsilon-3/2}L^2 \oplus \Pi_\theta \Pi_y^\perp x^{-\epsilon-1/2}L^2
\oplus \Pi_\theta^\perp x^{-\epsilon+1/2} L^2$.    Now suppose that $\eta \in \calA^{-\epsilon} \Omega^k$ and write 
$\eta=D\zeta + \gamma$ where $\gamma \in L^{2}\mathcal{H}^{k}$.   We can write $\zeta = \zeta_0 + \zeta'$ where $\zeta'\in \calA^{\epsilon}$ 
for any $0 < \epsilon < 1$ and $\zeta_{0}$ is in the kernel of $D_{00}$ and has polyhomogeneous expansion starting with the indicial roots $0$ and $1/2$, 
i.e., 
\begin{equation} \label{e:leading}
\begin{aligned}
 \zeta_{0}&=f_{0} + f_{12}\frac{dx}{x^{5/2}} \wedge \frac{dy_{1}}{x^{1/2}}+f_{13}\frac{dx}{x^{5/2}} \wedge \frac{dy_{2}}{x^{1/2}}\\
 &+ f_{24}\frac{dy_{1}}{x^{1/2}}\wedge x^{1/2}(d\theta + y_{1}dy_{2})+  f_{34}\frac{dy_{2}}{x^{1/2}}\wedge x^{1/2}(d\theta + y_{1}dy_{2})\\
 &+ f_{1234}\frac{dx}{x^{5/2}} \wedge \frac{dy_{1}}{x^{1/2}} \wedge \frac{dy_{2}}{x^{1/2}}\wedge x^{1/2}(d\theta + y_{1}dy_{2})\\
 &+ x^{1/2}f_{4} [x^{1/2}(d\theta + y_{1}dy_{2}) + \frac{dx}{x^{5/2}} \wedge \frac{dy_{1}}{x^{1/2}} \wedge \frac{dy_{2}}{x^{1/2}}]
 \end{aligned}
\end{equation}
where all the coefficients $f_{*}$ are smooth and of order 0 at infinity. 
We check that $\delta \zeta=0$. Indeed, this follows by writing
\[
\langle \delta \zeta, \delta \zeta \rangle = \langle \delta \zeta, \eta - d\zeta - \gamma \rangle 
\]
and integrating by parts. The boundary term reduces to
\[
\lim_{\epsilon \to 0} \int_{x = \epsilon} \zeta_{0}\wedge d*\zeta_{0};
\]
using ~\eqref{e:leading} and the vanishing rates of the coefficients, a short computation shows that this limit equals zero.
\end{proof}

We now relate this weighted de Rham space to the intersection cohomology of the (projective) compactification $X$ of $M$ obtained by 
collapsing each circle fiber at $x =0 $ to a point; this space $X$ is stratified, with principal top-dimensional stratum an open dense set, and 
the codimension two stratum $B = \bT^2$ which is the nilmanifold $Y$ with its $S^1$ fibers collapsed.  Any point $p \in B$ has a 
neighborhood diffeomorphic to a product $\mathcal{V}\times C_1(\bS^1)$ where $\mathcal{V}$ is an open ball in $\bT^{2}$ and $C_1(\bS^1)$ is 
the truncated cone over the circle.   We refer to the discussion in \cite{HHM} for a more detailed explanation of this stratification and for the 
definition and basic properties of the intersection cohomology spaces.

%=======
%Having now equated the Hodge cohomology with a weighted de Rham cohomology, we now relate this latter spae to intersection cohomology of
%the compactification $X$ of $M$ obtained by  collapsing each circle fiber at $x =0 $ to a point.  This space $X$ is stratified (in fact, it is a projective
%variety), with two strata: the principal stratum, which is the open dense set, and the codimension two stratum $B = \bT^2$ obtained by collapsing
%the $\bS^1$ fibers in the nilmanifold $Y$.    Any point $p \in B$ has a neighborhood which is diffeomorphic to a product $\mathcal{V}\times C_1(F)$
%where $\mathcal{V}$ is an open ball in $\bT^{2}$ and $C_1(F)$ is the truncated cone over $F=\bS^{1}$.   We refer to the discussion in \cite{HHM}
%for a more detailed explanation of this stratification and for the definition and basic properties of the intersection cohomology spaces we
%are about to introduce. 

Broadly speaking, the intersection homology of a stratified space is the homology of a complex of (say simplicial) chains for which the
dimensions of the intersections of their boundary faces with each stratum are regulated by a perversity function $\frakp$.  The dual complex
leads to the intersection cohomology, $\IH^*_{\frakp}(X)$, or in the present setting $\IH^*_{\frakp}(X,B)$ to emphasize the particular stratification.

%The intersection cohomology space $IH^{*}_{\frakp}(X,B)$ is associated to the intersection complex, where the perversity $\frakp=\{p(1),p(2)\}$ regulates the dimension of intersection of chains with strata. In our case $\ell=2$ and the regular perversity is given by $p(1)=p(2)=0$. 

Let us fix the `regular' perversity function $\frakp$, which has $\frakp(1) = \frakp(2) = 0$. 
We now define intersection cohomology with general perversity for $(X,B)$ as:
%\begin{definition}
%The intersection cohomology for the stratified space $(X,B)$ is defined as follows
\begin{equation}\label{e:IH}
\IH^{*}_{j}(X,B):=\left\{ 
\begin{matrix}
H^{*}(X-B), & j\leq -1\\
\IH_{\frakp}^{*}(X), & 0\leq j\leq \ell-2\\
H^{*}(X,B), & j\geq \ell-1
\end{matrix}\right.	
\end{equation}
The middle line $j=0$ corresponds to the regular perversity with $\frakp(1)=\frakp(2)=0$. 

We next relate the weighted cohomology for the Tian-Yau metric (which we assume equals $g_{\GH}$ near $x=0$) to that for the model metric
$$
g_{\mathrm{model}}=x g_{\GH} =  \frac{dx^{2}}{x^{4}} + g_{\bT^{2}} + x^{2} g_{\bS^{1}}. 
$$
The volume forms of $g_{\GH}$ and $g_{\model}$ are $x^{-3}\, dxdy_{1}dy_{2}d\theta$ and $x^{-1}\, dxdy_{1}dy_{2}d\theta$, respectively. 
%We use as an orthonormal basis for 1-forms for $g_{\model}$ the coframe $\{\tilde e_{i}=x^{1/2}e_{i}\}_{i=1}^{4}$ where the $e_{i}$ are as
%in~\eqref{1forms}.  If 
%\[
%\omega=\sum \omega_{i_{1}i_{k2}\dots i_{k}}\tilde e_{i_{1}}\wedge\dots \wedge \tilde e_{i_{k}} \in x^b L^2\Omega^k(M, \dvol_{g_{\model}})
%\]
%then 
%\[
%\int_0^1 |\omega_{i_1 \ldots i_k} |^2 x^{-2b-1}\, dx dy d\theta < \infty
%\]
%%$\omega \in \calO(x^{b})$ since the $L^{2}$ cut off for $g_{model}$ is 0. 
%On the other hand, rewriting $\omega$ in terms of the $g_{\TY}$-orthonormal coframe, we have
%$\omega=\sum x^{k/2}\omega_{i_{1}i_{k2}\dots i_{k}} e_{i_{1}}\wedge\dots e_{i_{k}}$, hence $\omega \in x^{a}L^{2}_{g_{TY}}$ implies 
%$x^{k/2}\omega_{*} \in \calO(x^{a+1})$ since the $L^{2}$ cutoff for $g_{TY}$ is $1$. Equating the two relations we get $b=a+1-\frac{k}{2}$, and hence
Taking this as well as the conformal change in the pointwise norm on $k$-forms into account, we calculate that
\begin{equation}
\WH^{k}(M, g_{TY}, a) = \WH^{k}(M, g_{\model}, a+1-\frac{k}{2}). 
\end{equation}

 The computation for $g_{\model}$ gives
 \begin{proposition}
 There is an isomorphism between the weighted cohomology and the intersection cohomology for $g_{\model}$
 \begin{equation}
 \WH^{k}(M,g_{\model},a)=\IH_{[a+1]}^{k}(X,B)
\end{equation}
assuming $a+k-\frac{3}{2}\neq 0$. 
\end{proposition}
\begin{proof}
The boundary of the stratified space $X$ is a codimension-2 strata $B$ that is isomorphic to $\bT^{2}$. 
%If we can show the following local computation is  correct for any neighborhood $\calU$ 
We now claim that
\begin{itemize}
\item If $\calU$ does not intersect $B$, then $\WH^{k}(\calU, g_{\model}, a)\simeq H^{k}(\calU, \mathbb{C})$;
\item If $\calU$ is a neighborhood intersecting $B$, then 
$$
\WH^{k}(\calU, g_{\model}, a) \simeq \IH^{k}_{\frakp}(\calU)=\left\{ 
\begin{matrix}
H^{k}(\bS^1) & k\leq \ell-2-\frakp(\ell)\\
0 & k\geq \ell-1-\frakp(\ell)
\end{matrix}, 
\right.
$$
\end{itemize}
If these are both the case, then by the sheaf-theoretic characterization of $\IH^*$, cf.\ ~\cite{GM2}, this will prove
the isomorphism between the appropriate weighted and intersection cohomology spaces. 

This equality is obvious if $\calU \cap B = \emptyset$, so we turn to the case where $\calU = \calV \times (0,\epsilon) \times \bS^1$ 
intersects $B$, where $\calV$ is an embedded disk in $\bT^{2}$.  By the K\"unneth formula,  
$$
\WH^{k}(\calU, g_{\model}, a) = \WH^k((0,\epsilon) \times \bS^1, g_{\redu}, a)
$$
where $g_{\redu} = \frac{dx^{2}}{x^{4}} +x^{2} d\theta^{2}$.

We  now compute this latter space.   Taking into account the difference in volume forms between $g_\redu$ and $dx^2/x^4$,
we see that
%First, any $ f \in \WH^{0}\left((0,\epsilon)\times F,  g_{\redu}, a\right)=\{f\in x^{a}L^{2}: df=0)\}$ satisfies $\del_{\theta} f= \del_{x}f=0$. 
%Converting the weight to $x^{\tilde a} L^{2}((0,\epsilon), \frac{dx^{2}}{x^{4}})$ 
%$$
%\frac{(x^{-a}f)^{2}}{x}dxd\theta <\infty \Leftrightarrow \frac{(x^{-\tilde a}f)^{2}}{x^{2}}dx <\infty
%$$
%we get that $\tilde a=a-\frac{1}{2}$. Therefore
$$
\WH^{0}\left((0,\epsilon)\times F,  g_{\redu} , a\right) = \WH^{0}((0,\epsilon), \frac{dx^{2}}{x^{4}}, a-\tfrac{1}{2}) \otimes H^{0}(\bS^1).
$$
while
%=======
%we get that $\tilde a=a-\frac{1}{2}$. Therefore
%$$
%WH^{0}\left((0,\epsilon)\times F,  \frac{dx^{2}}{x^{4}} +x^{2} d\theta^{2}, a\right) = WH^{0}\left((0,\epsilon), \frac{dx^{2}}{x^{4}}, a-\frac{1}{2}\right) \otimes H^{0}(F).
%$$
%A similar computation for $1-$ and $2-$forms shows the following:
%\begin{equation}
%\begin{aligned}
%WH^{0}((0,\epsilon)\times F,  \frac{dx^{2}}{x^{4}} +x d\theta^{2}, a)=	WH^{0}\left((0,\epsilon), \frac{dx^{2}}{x^{5}}, a-\frac{1}{4}\right) \otimes H^{0}(F)\\
%WH^{1}((0,\epsilon)\times F,  \frac{dx^{2}}{x^{5}} +x d\theta^{2}, a) = WH^{1}\left((0,\epsilon), \frac{dx^{2}}{x^{5}}, a-\frac{1}{4}\right) \otimes H^{0}(F)\\
%\oplus 
%WH^{0}\left((0,\epsilon), \frac{dx^{2}}{x^{5}}, a+\frac{1}{4}\right) \otimes H^{1}(F) \\
%WH^{2}((0,\epsilon)\times F,  \frac{dx^{2}}{x^{5}} +x d\theta^{2}, a) = WH^{1}\left((0,\epsilon), \frac{dx^{2}}{x^{5}}, a+\frac{1}{4}\right) \otimes H^{1}(F)
%\end{aligned}
%\end{equation}
%Rewriting the above computations we get

\begin{equation}
\begin{aligned}
\WH^{k} & \left((0,\epsilon)\times F,  g_{\redu} , a\right) \\
& =  \WH^{0}((0,\epsilon), \frac{dx^{2}}{x^{4}},  k+a-\tfrac{1}{2}) \otimes H^k(\bS^1) \\ 
& \qquad \oplus \WH^{1}((0,\epsilon),  \frac{dx^{2}}{x^{4}}, k-1+a-\tfrac{1}{2}) \otimes H^{k-1}(\bS^1)
\end{aligned}
\end{equation}
when $k = 1, 2$ (so long as the weights are nonindicial, so that all these weighted cohomologies are finite dimensional). 

It is now easy to check the following: 
%\begin{equation*}
\begin{align*}
& \WH^{0}\left((0,\epsilon), \frac{dx^{2}}{x^{4}}, \gamma\right)=\left\{ 
\begin{matrix}
0, & \gamma\geq -\frac{1}{2}\\
\mathbb{C}, & \gamma< -\frac{1}{2}
\end{matrix}
\right. \\
& \WH^{1}\left((0,\epsilon), \frac{dx^{2}}{x^{4}}, \gamma\right)
=%\left\{
%\begin{matrix}
0, \ \  \gamma\neq 0
%\text{ infinite dim}, & \gamma=0
%\end{matrix}
%\right.
\end{align*}
%\end{equation*}
(When $\gamma=0$, the omission of $\gamma = 0$ is because the range of $d$ is not closed then.) 
We obtain that if $a+k-\frac{3}{2}\neq 0$, then
$$
\WH^{k}(\calU, g, a) =\left\{
\begin{matrix}
H^{k}(\bS^1) & k <-a\\
0 & k\geq -a
\end{matrix}
\right.,
$$
and in fact, the same is true for the intersection cohomology, %for $\IH^{*}_{p}$ that
$$
\IH^{k}_{\frakp}=\left\{
\begin{matrix}
H^{k}(\bS^1) & k\leq \ell-2 - \frakp(\ell) \\
0 &k\geq \ell-1-\frakp(\ell)
\end{matrix}
\right.
$$
assuming that $\ell=2$, and $\frakp(2)=[a+1]$.
 \end{proof}
 
Converting back to the Tian-Yau metric, we have
\begin{equation}\label{e:TYIH}
\WH^{k}(M, g_{TY}, a)=\IH^{k}_{[a+2-\frac{k}{2}]}(X,B).
\end{equation}
 
Combining these results, we arrive at the 
 \begin{theorem}
If $g$ is a Tian-Yau metric on $M$, then there is an isomorphism
% There is an isomorphism between the $L^{2}$ cohomology of a Tian-Yau metric $(M,g)$ and the ordinary cohomology
\begin{equation}
L^{2}\calH^{k}(M) \simeq
 \left\{ 
\begin{matrix}
H^{k}(X,B) & k=0,1\\
\mathrm{Im}\, (\IH^{2}(X,B) \rightarrow \IH_{0}^{2}(X,B)) & k=2\\ 
\IH_{0}^{3}(X,B) & k=3\\
\mathrm{Im}\, (\IH_{0}^{4}(X,B) \rightarrow H^{4}(X-B)) & k=4
\end{matrix} 
 \right.
 \end{equation}
  \end{theorem}
\begin{proof}	
When $a$ is a sufficiently small positive number, Theorem~\ref{Th:weighted} equates
$$
L^{2}\mathcal{H}^{k}(M,g) \simeq \mathrm{Im}\, (\WH^{k} (M, g, a) \rightarrow \WH^{k} (M, g, -a)).
$$
The result follows by combining this with~\eqref{e:TYIH} and~\eqref{e:IH}. 
\end{proof}

We can go further than this since the projective compactification of any ALH* is a weak del Pezzo surface.  We divide the class of ALH* spaces
into the subclasses $\mathrm{ALH}^*_b$, where $b$ is the degree of the circle bundle of the nilmanifold $Y$. As stated in~\cite[Theorem 1.4]{hsvz2}, any $\mathrm{ALH}^{*}_{b}$ manifold can be compactified to a weak del Pezzo surface of degree $b$ by adding a smooth anticanonical elliptic curve. 
\begin{theorem}
If $M$ is of type $\mathrm{ALH}^*_b$, then its Hodge cohomology space $L^{2}\calH^k(M)$ is trivial except when $k=2$, and
$L^{2}\calH^{2}(M) \cong H^{2}(M)$ has dimension $11-b$. 
\end{theorem}
\begin{proof}

By Poincar\'e duality (or just the Hodge star), it suffices to consider only the cases $k = 0, 1, 2$. 

Any $L^2$ harmonic function must be constant, and nonzero constants do not lie in $L^2$, so $L^2\calH^0(M) = \{0\}$, agreeing
with the fact that $H^0(M,B) = 0$. 

Next, when $k=1$, consider the long exact sequence
$$
\dots \rightarrow H^{0}(M) \rightarrow H^{0}(\partial M) \rightarrow H^{1}(M,\partial M) \rightarrow H^{1}(M) \rightarrow \dots
$$
The manifold $M$ has only one end (this follows from the Cheeger-Gromoll splitting theorem), so the first map is surjective. In addition,
$M$ is simply connected, so $H^{1}(M)=0$ as well. Therefore $L^{2}\calH^1(X,B)=H^1(X,B)=H^1(M, \del M)=0$. 

Now let $k=2$. We first prove that the map $H^{2}(X,B) \to \IH_{0}^{2}(X,B)$ is surjective.   At the level of homology, the $1$- and $2$-chains satisfying
the perversity $0$ condition in $(X,B)$ do not intersect $B$, hence the map $I_{0}H_{2}(X,B) \to H_{2}(X-B)$ is injective, hence by duality,
$L^{2}\calH^{2}(M)=\IH_{0}^{2}(X,B)$ is surjective. 

Next, let $N(B)$ be a normal neighborhood of $B$ in $X$. The Mayer-Vietories sequence (see~\cite[\S 2.2.2]{HHM}) states that
$$
\begin{aligned}
&\dots \rightarrow \IH_{0}^{1}(X,B) \rightarrow H^{1}(M) \oplus \IH^{1}_{0}(N(B), B) \rightarrow H^{1}(\partial X)\\
& \rightarrow \IH^{2}_{0}(X,B) \rightarrow H^{2}(M) \oplus \IH_{0}^{2}(N(B), B) \rightarrow H^{2}(\partial X) \rightarrow \IH_{0}^{3}(X,B) \rightarrow\dots
\end{aligned}
$$

The first and last terms here vanish, i.e., $\IH_{0}^{1}(X,B)=\IH_{0}^{3}(X,B)=0$, using what we calculated for $1$-forms above. 
We next claim that $\IH_{0}^{i}(N(B), B)\simeq H^{i}(\del X)$ when $i=1,2$.  Indeed, when $i=1$, the $0$-perversity only allows for $1$-chains 
that do not intersect $B$, and $N(B)-B$ retracts to $\partial X$, hence $\IH_{0}^{1}(N(B),B)=H^{1}(\partial X)$.  When $i=2$,  the $0$-perversity $2$-chains 
may only intersect $B$ in points, while the boundaries of these cycles cannot intersect $B$. Any such $2$-chains can be deformed to a $2$-chain 
in $\del X$.  On the other hand, the only $0$-perversity $3$-cycle in generated by $\del X$ itself, so $H^{2}(\del X) \cong \IH_{0}^{2}(N(B),B)$. 
Using that $H^{1}(M)=0$, we obtain that $\IH^{2}_{0}(X,B)\simeq H^{2}(M)$. 

Finally, we quote that for $1\leq b\leq 9$, if $M$ is $\mathrm{ALH}^*_{b}$, then by the compactification to the degree-$b$ weak del Pezzo surface with a smooth anticanonical elliptic curve divisor~\cite{hsvz2}, one gets $\dim H^{2}(M) = 11-b$ (see~\cite{LeeLin} for example for the dimension calculation).   
%=======
%First of all, the first term and the last term vanish, i.e. $IH_{0}^{1}(X,B)=IH_{0}^{3}(X,B)=0$, by duality and the result of $k=1$ above.  Next we show $IH_{0}^{i}(N(B), B)\simeq H^{i}(\partial X)$ for $i=1,2$. For $i=1$, the 0-perversity only allows for 1-chains that do not intersect $B$, and $N(B)-B$ retracts to $\partial X$, therefore $IH_{0}^{1}(N(B),B)=H^{1}(\partial X)$. For $i=2, $ the 0-perversity 2-chains are allowed to intersect $B$ at points while the boundary of these cycles cannot intersect $B$. Any such 2-chains can deform to a 2-chain in $\partial X$ under the contraction map. On the other hand the only 0-perversity 3-cycle is generated by $\partial X$ itself, so the 2-cycles in $H^{2}(\partial X)$ maps bijectively to $IH_{0}^{2}(N(B),B)$. Together with the fact that $H^{1}(M)=0$, we obtain the isomorphism $IH^{2}_{0}(X,B)\simeq H^{2}(M). $ And for $1\leq b\leq 9$, when $M=ALH^{*}_{b}$ the dimension of $H^{2}(M)$ is $(9-b)$. 

\end{proof}

\section{Uniqueness and asymptotic regularity of Tian--Yau metrics in higher dimensions}
In this section we prove that Tian-Yau metrics are polyhomogeneous at infinity.   Throughout this discussion, $M = Y \setminus D$ is the complement 
of a smooth anti-canonical divisor $D$ in a compact weakly Fano manifold $Y$ of complex dimension $n \geq 2$, and $g$ is a Tian-Yau metric asymptotic to 
the Calabi model \eqref{C1}.  We show that $g$ has a polyhomogeneous expansion as $r \to \infty$ ($x \to 0$).   We make sense of this
as follows. In terms of the compactification of $M$ as a manifold with boundary $X$ defined in Section 2, the Calabi metric $g_C$ is obviously 
polyhomogeneous at $x=0$.  We prove here a relative statement, that if $\omega$ is the K\"ahler form for $g$ then $\omega - \omega_C$ is 
polyhomogeneous on $X$.   The paper \cite{hsvz} shows that if $\dim M =  4$, this is automatic in a very strong way, since this difference
is actually exponentially decaying as $r \to \infty$ (or in other words, decays like $e^{-\alpha/x}$ for some $\alpha > 0$).   There
are Tian-Yau metrics in higher dimensions which enjoy this same exponential decay, but as shown recently by Y.\ Chen \cite{YifanChen},
this is not always true.  A metric is called Calabi if this difference decays exponentially, and weakly Calabi if it decays polynomially; Chen
shows that the former is true only if a certain cohomological condition is satisfied. This condition is automatical when $n=2$, but Chen
shows that there is a rich class of so-called weakly Calabi Tian-Yau metrics, with $\omega - \omega_C = \calO(x^{1/n})$. To follow the notation in~\cite{YifanChen}, in this section we denote $r$ as the distance of $g$, which is of order $x^{-\frac{n+1}{2n}}$, and the variable $z$ used in~\cite{YifanChen} is related by $z=x^{-1/n}$. We prove here that
for these new metrics, this difference is polyhomogeneous, with leading exponent $1/n$.  This extra regularity is expected to have interesting 
applications, as in \cite{CMR}. 

We begin by reviewing the steps in her proof.  The Calabi K\"ahler form $\omega_C$ is only defined for $x < x_0$, say.  Chen first 
constructs a K\"ahler form $\omega_Y$ which is defined on all of $M$, and which converges exponentially to $\omega_C$.  Ideally
one could now apply what she terms the `Tian-Yau-Hein machine', namely, solving the Monge-Ampere equation to find
a perturbation of $\omega_Y$ which is Ricci flat.  The analysis needed here is now fairly standard, but it can be carried out only
when the Ricci form for this approximate solution decays like $r^{-(2 + \epsilon)}$ for some $\epsilon > 0$, and unfortunately, the
Ricci form for $\omega_Y$ decays only like $r^{-2/(n+1)}$. Her main step then is to construct a better approximate solution by
solving a finite sequence of linear Poisson equations on $M$ so as to produce successive correction terms, and hence to obtain a 
metric with Ricci form decaying faster than $r^{-2}$, as required. Finally the Tian-Yau-Hein method can then be applied to complete
the construction.

In this brief section we show how the analytic tools used in this paper show that this solution metric is polyhomogeneous. 
The key tool is the following analytic result, which is the direct analogue for the scalar Laplacian of the Proposition~\ref{prop:poly} above. 
\begin{lemma}\label{lemmaA}
If $\Delta u=f$ where $f\in \calA^a(M)$ and 
\[
u\in x^{c-(\frac{2}{n}+3)}\Pi_{\theta}\Pi_{y} H^2_{\bfa} (M)  \oplus x^{c-(\frac{2}{n}+1)}\Pi_{\theta}\Pi_{y}^{\perp} H^2_{\bfa} (M) \oplus 
x^{c-(\frac{2}{n}-1)}\Pi_{\theta}^{\perp}\Pi_{y}^{\perp} H^2_{\bfa}(M),
\]
for some $c < a$, and if neither $a$ nor $c$ are indicial weights, then 
\[
u \in \Pi_{\theta}\Pi_{y} \calA^E_{\phg}(M) + \calA^a(M),
\]
where $E$ is the set of indicial roots of $\Delta$ lying in the interval $(c,a)$.   In other words, $u$ admits a partial polyhomogeneous expansion, but
only up to order $x^a$.  If $f$ is itself polyhomogeneous or rapidly vanishing as $x \to 0$, then $u$ is polyhomogeneous to all orders. 
\end{lemma}

As an initial application of this result, we note that each of the correction terms in Chen's approximate solution are obtained by solving an
equation of the form $\Delta u = f$ with $f$ already known to be polyhomogeneous.  Applying the final statement of Lemma~\ref{lemmaA},
we see that these correction terms are polyhomogeneous, and hence that our approximate solution has K\"ahler form $\omega_0$ with 
$\omega_0 - \omega_C$ is polyhomogeneous, with leading term of order $x$ (and with Ricci form decaying like $x^{2 + \epsilon}$). This 
can be achieved by using $\beta+ i\del\delbar U_K$ with a sufficiently large $K$ in~\cite[Theorem 6.1]{YifanChen}.
The exact solution is then given by $\omega_0 + i \del \delbar u$ where 
\begin{equation}
\frac{ (\omega_0 + i \del \delbar u)^n}{\omega_0^n} = e^{f_0},
\label{MA1}
\end{equation}
where $f_0$ is a priori known to be polyhomogeneous.  Thus our result follow from the 
\begin{theorem}
Suppose that $f\in \calA_{\phg}^{\delta}(M)$ for some $\delta > 0$ and let $u$ be the unique decaying solution to the equation \eqref{MA1}.
Then $u \in \calA_{\phg}^G$ for some positive index set $G$.
\end{theorem}
We present an abbreviated proof since this type of argument appears in several places already, see \cite{CMR, JMR}. 
\begin{proof}   The steps of the proof are as follows. We first note that by local elliptic regularity, the a priori bound on $u$
gives that $u \in x^\kappa H^\infty_{\bfa}(M)$ for any $\kappa < \delta - 1$. In other words, $u$ is smooth with respect to
repeated differentiations by the vector fields in $\calV_{\bfa}$.  We next upgrade this and show that $u$ is conormal, i.e., 
infinitely regular with respect to differentiations by $b$-vector fields.  The final step is to show that it has a polyhomogeneous
expansion, as claimed.

The first of these steps follows by first recalling that since the volume form of $g_C$ is for any $n$ boundedly equivalent to $x^{-3} dx dy d\theta$,
(where $y$ is now a local coordinate on the divisor $D$), the bound $u = \calO(x^\delta)$ implies that $u \in x^\kappa H^\infty_{\bfa}(M)$ for any 
$\kappa < \delta - 1$.  The local elliptic estimates for this complex Monge-Ampere equation then control the corresponding weighted
norms of the functions obtained by taking arbitrarily many derivatives of $u$ with respect to vector fields which are unit length
with respect to the Calabi metric.  However, any such vector field is of the form $V = x^{1/n - 1} W$ where $W$ is a `size $1$' vector field
for the metric $g_{\bfa}$.  Clearly then, the higher derivatives of $a$ with respect to $\calV_{\bfa}$ are all controlled, hence
$u \in x^\kappa H^\infty_{\bfa}$ with $\kappa$ as above. 

As for the second step, we start with this first assertion along with the estimate $i\del\delbar u \in \calO(x)$, as proved
by Chen. Expanding the equation~\eqref{MA1} leads to the equivalent equation
$$
\Delta_{\omega_{0}} u=E(u) + (e^{f_{0}}-1), \ E(u) =\sum_{j=2}^{n} \frac{(n-j)!j!}{n!}\frac{(\omega_{0}^{n-j}\wedge (i\partial \bar \partial u)^{j})}{\omega_{0}^{n}}.
$$
Now apply the left parametrix $G_{l}$ of $\Delta_{\omega_{0}}$ to this equality to get that
$$
u + R_{\ell} u = G E(u) + G (e^{f_{0}}-1).
$$
Choose any $b$-vector field $V$, i.e., $V \in \calV_b$ and compute that 
$$
Vu  = - V R_{l}u + G V E(u) - [G, V] E(u) + V G(e^{f_0}-1).
$$
Recalling that the error terms $R_\ell$ and $R_r$ are polyhomogeneous on the original (un-blown-up) space $X^2$, and noting also that since 
$f_0$ is polyhomogeneous, so is $V G (e^{f_0} - 1)$, we see that we need only focus on the second term on the right, which we write out as 
$$
G V E = GV \sum_{j=2}^{n} \frac{(n-j)!j!}{n!}\frac{(\omega_{0}^{n-j}\wedge (i\partial \bar \partial u)^{j})}{\omega_{0}^{n}}
$$
Since $i\del\delbar u \in x^\kappa H_{\bfa}^{\infty}(M)$ and $|i\del\delbar u|\leq Cx$, it follows that $E(u) \in x^{2\kappa}H^\infty_{\bfa}$, with
$|E(u)| \leq C x^2$.  Using $G V E(u) = G (x^{2}V)  (x^{-2}E(u))$ and that $x^{2}V \in \calV_{\bfa}$,  this term remains bounded.

The third term is a bit more complicated because of the commutator.  We claim that 
$[G,V] \in \Psi^{2,\calE}_a(M)$. To study this commutator, apply $G \Delta_{\omega_0}$ to the left and $\Delta_{\omega_0} G$ 
to the right. This gives 
$$
\begin{aligned}
(Id + R_l)[G,V] (Id + R_r) & = G \Delta_{\omega_0}[G,V]\Delta_{\omega_0} G\\
& = G[\Delta_{\omega_0}, V] G  + G R_r V \Delta_{\omega_0} G - G \Delta_{\omega_0} V R_l G
\end{aligned}
$$ 
Picking out the commutator term, we have
$$
\begin{aligned}
[G, V] & = G[\Delta_{\omega_0}, V] G  + G R_r V \Delta_{\omega_0} G \\ 
& - G \Delta_{\omega_0} V R_l G - R_l [G,V]- [G,V]R_r - R_l[G,V]R_r.
\end{aligned}
$$
It is straightforward to check that $[\Delta_{\omega_0}, V]\in \operatorname{Diff}^2_a(M)$; this is not entirely obvious since $V$ lies
only in $\calV_b$.  All of the remaining terms are polyhomogeneous. This shows  that the entire right hand side lies in 
$\Psi_\bfa^{-2, \calE}(M)$. Therefore $[G, V] E_0$ is also bounded, and hence the right hand side lives in $x^\kappa L^2$.
Iterating this argument, applying further $b$-vector fields inductively, we obtain that $u \in \calA^{\delta}(M)$. 

We come finally to the last step. Observe that %\eqref{MA1} 
$$
\Delta_{\omega_{0}} u=E(u) + (e^{f_{0}}-1), \ E(u) =\sum_{j=2}^{n} \frac{(n-j)!j!}{n!}\frac{(\omega_{0}^{n-j}\wedge (i\partial \bar \partial u)^{j})}{\omega_{0}^{n}} \in \calA^{2}(M).
$$
The linearized operator for this equation is $L=\Delta_{\omega_{0}}$.  Since $e^{f_{0}}-1 \in \calA_{\phg}^{2}(M)$, we obtain
that $u \in \calA^1_{\phg} + \calA^2$.   Now insert this back into the equation to get, by the same reasoning, that
$u \in \calA^1_{\phg} + \calA^3$.   Iterating this procedure leads to the full expansion.
\end{proof}

\section{Deformations of Tian-Yau metrics}

\subsection{Deformations of hyperK\"ahler triples}

There are various ways to study the local deformation theory of hyperK\"ahler manifolds, and we adopt here the elegant ``hyperK\"ahler triples'' formulation
due to Donaldson, cf.\ \cite{Donaldson}, \cite{Fos}, \cite{SchroersSinger}. . We recall this and then explain how it adapts to the ALH* setting.

A triplet $\ulom = (\omega_1, \omega_2, \omega_3)$ of symplectic $2$-forms on an oriented $4$-manifold $M$  is called a symplectic triple 
if the associated matrix of $4$-forms with entries $\omega_i \wedge \omega_j$ is positive definite.  In other words, relative to any fixed volume 
form $dV$, the symmetric matrix $A$, where $\omega_i \wedge \omega_j = A_{ij}\, dV$, is positive definite.  Any such triple determines another 
volume form $\frac13( \omega_1 \wedge \omega_1 + \omega_2 \wedge \omega_2 +\omega_3 \wedge \omega_3)$, and we say that $\ulom$ 
is a hyperK\"ahler triple if the matrix $A$ relative to this preferred volume form is the identity, i.e., 
\[
Q_{ij} := \omega_i \wedge \omega_j - \frac13( \omega_1 \wedge \omega_1 + \omega_2 \wedge \omega_2 + \omega_3 \wedge \omega_3) \delta_{ij} = 0
\]
for all $i, j$. Invariantly, $Q$ is a map from the space of symplectic triples to $\mathrm{Sym}^2_9(\RR^3) \otimes \Omega^4(M)$, the space
of symmetric trace-free $3$-by-$3$ matrices with entries in $\Omega^4(M)$. 

A hyperK\"ahler triple $\ulom$ uniquely determines a hyperK\"ahler metric $g$ on $X$, see~\cite{SchroersSinger}. Conversely, 
if $g$ is a hyperK\"ahler metric, then the bundle $\Lambda^2_+(M)$ of self-dual 2-forms on $M$ is a parallel subbundle of $\Lambda^2(M)$, and 
any orthonormal parallel frame $\omega_1, \omega_2, \omega_3$ for the space of sections of this subbundle constitutes a hyperK\"ahler triple. 
(Such a parallel frame exists globally if $M$ is simply connected, which will be the case for all spaces considered here.) The ambiguity in this 
correspondence is an element of $\mathrm{SO}(3)$. 

If $\ulom$ is a hyperK\"ahler triple, and if $\uleta$ is a triple of (smooth) closed forms which are small in $\calC^0$, then $\ulom + \uleta$ 
remains a symplectic triple, and is a hyperK\"ahler triple if $Q(\ulom + \uleta) = 0$.  This equation has a lot of redundancy. Indeed, 
\[
Q(\ulom + \uleta, \ulom + \uleta) = Q(\ulom, \ulom) + 2 Q(\uleta, \ulom) + Q(\uleta, \uleta),
\]
where $Q(\ulom, \ulom) = 0$  and the middle term is the polarization
$$
2Q(\uleta, \ulom)_{ij}=\eta_{i}\wedge \omega_{j} + \eta_{j} \wedge \omega_{i}-\frac{1}{3}\sum_{k=1}^{3} (\omega_{k}\wedge \eta_{k}+ \eta_{k}\wedge \omega_{k})\delta_{ij}.
$$
Decompose each $\eta_i$ into self-dual and anti-self-dual parts $\eta_i^+ + \eta_i^-$; since each $\omega_i$ is self-dual, 
$\eta_i \wedge \omega_j = \eta_i^+ \wedge \omega_j$, i.e., the $\eta_i^-$ do not contribute to this wedge product. Now impose the more 
restrictive equations 
%$\calJ(\uleta) = 0$, where 
\begin{equation}
\calJ(\uleta)_{ij} := 2\eta_i^+ \wedge \omega_j  + Q(\uleta, \uleta)_{ij} = 0,\ \ i, j = 1, 2, 3.
\label{gauged}
\end{equation}
Thus if $\uleta$ is a solution, and we write $\eta_{i}^{+}=\sum_{k} b_{ik}\omega_{k}$, then $B = (b_{ij})$ must  be trace-free and symmetric since the
same is true for $Q$. 

We claim that \eqref{gauged} fixes a gauge.  The infinitesimal $\mathrm{SO}(3)$ action on the triple $\ulom$ produces an infinitesimal deformation
$\uleta^+ = B \ulom$ where $B$ is skew-symmetric, while the infinitesimal conformal rescaling yields an $\uleta^+ = \lambda \ulom$, i.e.,
this $B$ is a multiple of the identity. Thus \eqref{gauged} kills both of these families. 

Since $\ulom$ is a basis of self-dual $2$-forms, the nullspace of the linearization 
\[
D\calJ_{ij}|_{\uleta = 0}(\dot{\uleta}) = \dot{\eta}_i^+ \wedge \omega_j ; 
\]
consists of forms with $\dot{\eta}^+_i = 0$, $i = 1, 2, 3$, i.e., $\dot{\uleta}$ is anti-self-dual.  Since $\dot{\eta}$ is closed, and a closed anti-self-dual 
form is harmonic, we conclude that the space of infinitesimal solutions to \eqref{gauged} is just $\mathcal H^2_-$, the space of anti-self-dual harmonic $2$-forms.  

Going further, \eqref{gauged} also fixes a slice transverse the diffeomorphism group.  To see this, recall that any one-parameter family of diffeomorphisms 
is generated by a vector field $V$, and the infinitesimal action on the hyperK\"ahler triple is the Lie derivative, $\calL_V \ulom$. By Cartan's formula and 
the fact that $\ulom$ is closed, $\calL_V \, \ulom=d (\iota_V \ulom)$. We denote the map $V \mapsto \calL_V \ulom = d \iota_V \ulom$ by
\[
L: \calC^\infty(X, TX) \rightarrow \Omega^{2}(X) \otimes \RR^{3}. 
\]
Making the various obvious identifications, this has formal adjoint 
\[
L^{*}: \Omega^{2}(X) \otimes \RR^{3} \rightarrow \calC^\infty(X, TX), \ \uleta\mapsto (\ulJ\cdot (d^{*} \uleta))^{\#},
\]
where $\ulJ = (J_1, J_2, J_3)$ is the hyperK\"ahler family of complex structures and $\ulJ\cdot \uleta=\sum J_{i}\eta_{i}$.   Note that if 
$\uleta^{-}\in \calH_{-}^{2} \otimes \RR^{3}$, then $d^{*}\uleta^{-}=0$ and hence $L^{*}(\uleta^{-})=0$.  Hence, at least when $X$ is compact,  tangent
vectors to the solution manifold $\calJ^{-1}(0)$ are orthogonal to the diffeomorphism orbits.  We examine this later in our noncompact setting.

As just shown, nonzero solutions to $\calJ(\uleta) = 0$ have nonvanishing self-dual components $\eta_i^+$, and by Hodge theory we can write
\[
\eta_i^+ = d^+ a_i + \zeta_i,
\]
where $a_i \in \Omega^1(X)$ and $\zeta_i$ is a self-dual harmonic $2$-form.   Since $Q(\uleta^\pm, \uleta^\mp) = 0$, 
\[
Q(\uleta^+ + \uleta^-, \uleta^+ + \uleta^-) = Q(\uleta^+, \uleta^+) + Q(\uleta^-, \uleta^-),
\]
and thus, regarding $\uleta^-$ as the `input parameter', we seek solutions $\ula$, $\ulzeta$ to
\begin{equation}\label{e:da}
(d^+ a_i + \zeta_i) \wedge \omega_j  = -\frac12 Q_{ij}( d^+ \ula + \ulzeta, d^+ \ula + \ulzeta) - \frac12 Q_{ij}( \uleta^-, \uleta^-).
\end{equation}
We obtain an elliptic equation if we supplement \eqref{e:da} with $d^* a_i = 0$, $i = 1, 2, 3$. Since the left hand 
side of \eqref{e:da} specifies the components of each $d^+ a_i + \zeta_i$ relative to the basis $\ulom$, the entire system can be expressed as
\begin{equation}
D \ula + \ulzeta = -\frac12 \mathcal Q( d^+ \ula + \ulzeta, d^+ \ula + \ulzeta) - \frac12 \mathcal Q( \uleta^-, \uleta^-)
\label{HKTeqn}
\end{equation}
where $D = d^* + d^+: \Omega^1(X) \longrightarrow \Omega^0(X) + \Omega^2_+(X)$. 

Finally, writing $a_i = D^* u_i$ for $u_i \in \Omega^0 + \Omega^2_+$, then \eqref{HKTeqn} takes the form
\begin{equation}
DD^* \ulu +  \ulzeta = F( DD^*\ulu, \, \ulzeta, \, \uleta^-),
\label{fHKTeqn} 
\end{equation} 
where $F$ just represents the nonlinear operator on the right. For hyperK\"ahler metrics, the Weitzenb\"ock formula reduces to $DD^* = \nabla^* \nabla$, 
and using the parallel trivialization
of $\Lambda^2_+$, this reduces even further to four copies of the scalar Laplacian.  We shall write
\begin{equation}\label{eq:ui}
u_{i}=f_{i}^{0} + \sum_{j=1}^{3}f_{i}^{j}\omega_{j} \in \Omega^{0} + \Omega_{+}^{2};
\end{equation}  
each $f_{i}^{j}$ is well-defined up to a harmonic function (which is automatically constant if $X$ is compact.) 

In the compact setting, these calculations are easily justified, and lead immediately to the following well-known result, see~\cite{Donaldson, FLS} for further details.
\begin{proposition}
Let $(X,g)$ be a compact hyperk\"ahler manifold. Then the local hyperk\"ahler deformation theory is unobstructed. A neighborhood of $(X,g)$ in the space of 
all hyperk\"ahler metrics modulo the aforementioned gauge actions is diffeomorphic to a neighborhood of $0$ in $\calH^2_-(X,g) \otimes \RR^3$.
\label{cpctcase}
\end{proposition}
\begin{proof}
Consider the map 
\[
\begin{split}
\calK: \left((\calC^{2,\alpha}(X)/\RR)^4 \oplus \calH^{2}(X,g) \right) \otimes \RR^3   \longrightarrow \calC^{0,\alpha}(\Omega^{2}(X))\otimes \RR^3
\end{split}
\]
given by
\[
(\ulu,  \ulzeta, \uleta^-) \longmapsto \calK(\ulu, \ulzeta, \uleta^-)  = DD^*\ulu + \ulzeta - F(DD^*\ulu, \, \ulzeta,\ \uleta^-),
\]
where $(\ulzeta, \uleta^-) \in (\calH^2_+ \oplus \calH^2_-)\otimes \RR^3 = \calH^2 \otimes \RR^3$.
Then
$$
D\calK|_{(0,0,0)}(\underline v, \underline \xi,0) = D D^{*} \underline v + \underline \xi. 
$$
Note too  that $DD^{*}$ decomposes as $12$ copies of the scalar Laplacian.    We see, therefore, that 
\[
D\calK_0 \longrightarrow (\calC^{2,\alpha}(X)/\calH^{0}(X))^4 \oplus( \calH^{2}_+(X,g)\otimes \RR^3
\]
is an isomorphism. Now apply the implicit function theorem to get that the solution space to $\calK = 0$ is locally parametrized by a 
neighborhood of the origin in $(\calH^{2}_-(X,g) \otimes \RR^3)$. 

The preceding discussion shows that the tangent space lies in the kernel of $L^{*}$, hence is orthogonal to the underlying diffeomorphism orbit through
$\ulom$, and is also transverse to the rotation and dilation action on $\Omega^2_+$. Therefore the moduli space is unobstructed and locally diffeomorphic to a 
neighborhood of $0$ in $(\calH^{2}_-(X,g) \otimes \RR^3)$. 
\end{proof}

\subsection{New features in the $ALH^*$ setting}
The development of the deformation theory of hyperK\"ahler structures in the previous subsection is fairly straightforward, but its adaptation 
to our noncompact setting presents some new challenges.  Some of this is relatively direct, using the mapping properties for the Laplacian 
for ALH* spaces described earlier, but there are new somewhat intricate considerations.
We first check that the various `formal' operations, e.g., integration by parts, carry over to this new setting, and with respect to the function 
spaces used here.  The infinitesimal deformations in this ALH* setting arise in two different ways: some occur due to
changes in the `internal moduli', and these decay at infinity and may be analyzed very much as in the compact setting. Others correspond to varying
the nilmanifold structure at infinity.  Both correspond to anti-self-dual harmonic $2$-forms: the former to forms which lie in $L^2$ while
the latter to forms with nonzero limits at infinity.
 
To understand the deformations at infinity, we first consider the family of deformations of $\ulom=(\omega_{1}, \omega_{2}, \omega_{3})$ 
given by 
$$
\tilde \omega_{i}(t)= \sum_{j=1}^{3} A_{ij}(t) \omega_{j}^{+} +\sum_{j=1}^{3} B_{ij}(t) \omega_{j}^{-}, i=1,2,3,
$$
where  $\{\omega_{j}^{\pm}\}_{j=1}^{3}$ are the bases for self-dual and anti-self-dual two-forms at infinity, 
\begin{equation}
\begin{aligned}
\omega_{1}^{\pm}=dr\wedge \theta \pm  r e_{1}\wedge e_{2}, \\
\omega_{2}^{\pm} = e_{1}\wedge \theta \pm  r e_{2}\wedge dr, \\
\omega_{3}^{\pm}= e_{2}\wedge \theta \pm r dr\wedge e_{1}.
\end{aligned}
\label{sdtf}
\end{equation}
(Here $r = 1/x \to \infty$ on the end.) We write $A(t)=(A_{ij}(t)), \ B(t)=(B_{ij}(t))$, where $A(0) = \mathrm{Id}$, $B(0) = 0$. 
To maintain that the $\tilde \omega_i$ be closed, we impose that $B_{i1}(t) \equiv 0$, $i=1,2,3$, which reflects the fact 
that $d\omega_{1}^{-}\neq 0$.

The equation stating that $\underline{\tilde{\omega}}$ is a hyperK\"ahler triple are
$$
\tilde \omega_{i}(t) \wedge \tilde \omega_{j}(t)=\delta_{ij} \frac13 \sum \tilde \omega_{\ell}\wedge \tilde\omega_{\ell}.
%Vol_{\ulom}, \ \mbox{where}\ \lambda(t)Vol_{\ulom} = \frac{1}{3}
$$ 
Since it is convenient to use a fixed volume form $Vol_{\ulom} = \frac{1}{3}\sum \omega_{\ell}^{+} \wedge \omega_{\ell}^{+}$, we define the scalar quantity $\lambda(t)$ by
\[
\frac13 \sum \tilde \omega_{\ell}\wedge \tilde\omega_{\ell} = \lambda(t) \frac{1}{3}\sum \omega_{\ell}^{+} \wedge \omega_{\ell}^{+}.
\]
Expanding the $\tilde{\omega}_i$ as above, this relationship between volume forms yields that
\begin{equation}\label{algebra}
AA^{T} - BB^{T}=\lambda(t) \mathrm{Id}.
\end{equation}
Furthermore, the gauge-fixing condition~\eqref{gauged} becomes 
\begin{equation}
A=A^{T},\ \ \mathrm{tr}(A-\mathrm{Id})=0. 
\label{gf}
\end{equation}
Taking the derivative of~\eqref{algebra} at $t=0$ gives
\begin{equation}
\dot A+\dot A^{T}=\dot \lambda \, \mathrm{Id}.
 \end{equation}
Using \eqref{gf}, we see immediately that 
$$
\dot A=0, \ \ \dot \lambda=0.
$$ 
In other words, if the anti-self-dual input at infinity is trivial, then the hyperK\"ahler triple is fixed up to rotation and overall scaling.  We obtain
one further equation by taking second derivatives:
\begin{equation}
\ddot A+\ddot A^{T} -\dot B\dot B^{T}=\ddot \lambda \, \mathrm{Id}. 
\label{mm}
\end{equation}

Before proceeding, let us record some facts about this parameter space
\begin{equation}
\calP =\{(A,B, \lambda): A^2 - BB^{T}=\lambda \, \mathrm{Id}, \ \mathrm{tr}(A)=3, \ B_{1i}=0, \ i = 1, 2, 3\},
\end{equation}
where $A$ is assumed to be symmetric, but $B$ is not.
\begin{lemma}
The space $\calP$ is a $6$-dimensional manifold in a neighborhood of the point $(\mathrm{Id}, 0)$, with
\[
T_{(\mathrm{Id},0,1)} = \{ \dot B: \dot B_{1i} = 0,\ i = 1, 2, 3\}.
\]
\end{lemma}
\begin{proof}
Writing $\calF(A,B,\lambda) = A^2 - BB^T - \lambda \, \mathrm{Id}$, we calculate that
\[
D\calF_{(\mathrm{Id}, 0, 1)}( \dot A, \dot B, \dot \lambda) = 2 \dot A - \dot \lambda \, \mathrm{Id};
\]
this is clearly surjective onto the space of symmetric $3$-by-$3$ matrices. Since $A$ is a priori symmetric and trace-free, the
nullspace equals the space of $\dot B$ with all $\dot B_{1i} = 0$.  The claim follows from the implicit function theorem.
\end{proof} 

As one further point, the relationship \eqref{mm} exhibits that we can regard $\ddot \lambda$ and the symmetric trace-free matrix
$\ddot A$ as functions of $\dot B$, and as such, 
\[
| \ddot \lambda| + |\ddot A| \leq C | \dot B|^2 \ \ \mbox{as}\ \ |\dot B| \to 0.
\]
\medskip

This lemma suggests that there should be a six-dimensional family of variations of the hyperK\"ahler structure at infinity, the infinitesimal versions
of which fill out the nullspace of $D\calF$.   In fact, as we now show, there are two separate three-dimensional families of variations, the first
corresponding to changes of complex structure in the Calabi Ansatz, and the second to changes of the associated semiflat metrics. 
While these appear to be quite different, we quote a result in \cite[Appendix A]{CJL} which shows that in fact these two families are `the same',
differing only by a family of hyperK\"ahler rotations.  In other words, up to this gauge equivalence, there is really only a three-dimensional family of
fundamentally distinct variations at infinity. 

We now describe the geometric variations corresponding to each of these three-parameter families.

\bigskip

\noindent {\large{\bf Variations of the Calabi ansatz: }}
There are two separate ways to vary the Calabi model, the first by changing the relative scaling of the divisor, which is a one-dimensional complex torus,
and the $S^1$ fiber, and the second by varying the conformal modulus of the torus.  We describe the corresponding infinitesimal parameters for each,
i.e., the associated values of $\ddot{A}$, $\dot{B}$ and $\ddot{\lambda}$

\bigskip

\noindent {\bf Variation by relative scaling:}    First consider the one-parameter family of coordinate rescalings 
\[
\tilde r=a(t) r , \ \ \tilde \theta = b(t) \theta, \ \ \tilde e_{1}=c(t)e_{1}, \ \ \tilde e_{2}=c(t) e_{2}, 
\]
where $a(t)$, $b(t)$ and $c(t)$ are functions which take the value $1$ at $t=0$. Since $e_1$ and $e_2$ are dilated by the same amount,
this amounts to a change of the relative scaling of the torus and $\bS^1$ fiber. The new $2$-forms $\tilde{\omega}_i(t)$ defined 
as in \eqref{sdtf} satisfy $\underline{\tilde{\omega}} = A \ulom^+ + B \ulom^-$ where $A$ and
$B$ are diagonal, with 
\begin{align*}
& A_{11} = \frac12 a (b + c^2),\ A_{22} = A_{33} = \frac12 c (b + a^2), \\
& B_{11} = \frac12 a(b - c^2),\ B_{22} = B_{33} = \frac12 c (b - a^2). 
\end{align*}
The condition $B_{11} = 0$ implies that $b = c^2$, so 
\[
A_{11} = ac^2, \ A_{22} = A_{33} = \frac12 c (a^2 + c^2),\ \ B_{22} = B_{33} = \frac12 c (c^2 - a^2).
\] 

Since $\dot{A}_{11} = 0$, we get that  $(a c^2)\dot{} = 0$, or equivalently, $\dot{a}  = - 2 \dot{c}$, so in particular, $a(t) c(t)^2 = 1 + \calO(t^2)$. 
We may normalize the overall scaling factor by assuming that $a(t) c(t)^2 = 1 - t^2$.   With this, the equation $\mathrm{tr}\, A = 3$ implies that
$c^3 + a^2 c = 2 + t^2$.  However, $ a^2 c^4 = (1-t^2)^2 = (a^2 c) c^3$, so
\[
c^3 + \frac{(1-t^2)^2}{c^3} = 2 + t^2.
\]
Solving this as a quadratic for $c^3$, then taking the cube root, gives finally that
\[
c(t) = \left( 1 + \frac12 t^2 + \frac{t}{2} \sqrt{12 - 3t^2} \right)^{1/3},\ \ a(t) = \frac{1-t^2}{c^{2}(t)}.
\]
 
Replacing $a(t)$ and $c(t)$ by $a(\alpha t)$ and $c(\alpha t)$, respectively, the corresponding infinitesimal variations of $\Lambda^2_\pm$ are then given by 
\begin{equation}\label{eq:calabi1}
\ddot A=
\begin{bmatrix}
-2\alpha^{2} & 0&0\\
0 & \alpha^{2} &0\\
0 & 0 & \alpha^{2}
\end{bmatrix}, \ \ 
\dot B=
\begin{bmatrix}
0&0 & 0\\
0& \sqrt{3}\alpha & 0\\
0&0 & \sqrt{3}\alpha
\end{bmatrix}, \ \ \ \ddot \lambda=-2\alpha^{2}.
\end{equation}

\bigskip

\noindent {\large{\bf Variation of the conformal modulus:} }
This is a two-dimensional family of deformations, which we write as
$$
\tilde e_{1}=a_{0} e_{1}+ c_0 e_{2}, \tilde e_{2}=b_{0}e_{2}+ c_0 e_{1}, \ \ \tilde r=(a_{0}b_{0}-c_0^{2})^{1/2} r, \ \tilde \theta = (a_{0}b_{0}-c_{0}^{2}) \theta.
$$
The matrix on $2$-forms is already symmetric, and the condition on its trace becomes
$$
(a_{0}b_{0}-c_{0}^{2})\left( (a_{0}b_{0}-c_{0}^{2})^{1/2} + a_{0}+b_{0} \right)=3.
$$
The two-dimensions of the family arise by letting these coefficient functions $a_0$, $b_0$ and $c_0$ depend on $\alpha t$ and $\beta t$.  
To simplify the matrices in~\eqref{eq:calabi2} below, it is convenient to choose
$c_{0}=\beta t$ and $a_{0}b_{0}-c_{0}^{2}=(1- \frac{\alpha^{2}+\beta^{2}}{9} t^{2})^{2}$.

Writing $s=3(1- \frac{\alpha^{2}+\beta^{2}}{9} t^{2})^{-2} -(1- \frac{\alpha^{2}+\beta^{2}}{9} t^{2})$, then the trace equation above is
equivalent to $a_{0}+b_{0}=s$, so 
$$
\begin{aligned}
a_{0}=\frac{1}{2}\left(s + \sqrt{s^{2}-4 \left((1- \frac{\alpha^{2}+\beta^{2}}{9} t^{2})^{2}+\beta^{2}t^{2}\right)} \right)\sim 1+\alpha t +\frac{7}{18} (\alpha^{2}+\beta^{2}) t^{2} \\
b_{0}=\frac{1}{2}\left(s - \sqrt{s^{2}-4 \left((1- \frac{\alpha^{2}+\beta^{2}}{9} t^{2})^{2}+\beta^{2}t^{2}\right)} \right) \sim 1-\alpha t +\frac{7}{18} (\alpha^{2}+\beta^{2}) t^{2} .
\end{aligned}
$$

Finally, the associated variations of $\Lambda^2_\pm$ are 
\begin{equation}\label{eq:calabi2}
\begin{aligned}
& \ddot A=
\begin{bmatrix}
-\frac{1}{3}(\alpha^{2}+\beta^{2}) & 0&0\\
0 &\frac{1}{6}(\alpha^{2}+\beta^{2}) &0\\
0 &0&\frac{1}{6} (\alpha^{2}+\beta^{2})
\end{bmatrix}, \\
& \quad \dot B= 
\begin{bmatrix}
0&0 & 0\\
0&\alpha &\beta\\
0&\beta&-\alpha
\end{bmatrix}, \ \ 
 \ddot \lambda=-\frac{2}{3}(\alpha^{2}+\beta^{2}).
 \end{aligned}
\end{equation}

\bigskip

\noindent {\large{\bf Variations of the semiflat model:}  }
We next provide a similar description of the infinitesimal variations of the semiflat model metric.  This section follows closely Appendix A in \cite{CJL}.
The notation in that paper is different than the one we have been using.  The precise correspondence between their variables and the ones used here is given
by the equations
\begin{equation}\label{e:CJLCoordinate}
\xi_{1}=c_{\tau}y_{2}, \ \xi_{2}=c_{\tau}y_{1}, \psi=\theta+\frac{1}{2} y_{1}y_{2}, \ \ell=r
%\xi_{1}=c_{\tau}\frac{y_{2}-y_{1}}{\sqrt{2}}, \ \xi_{2}=c_{\tau}\frac{y_{1}+y_{2}}{\sqrt{2}}, \psi=\theta+\frac{1}{2} y_{1}y_{2}, l=z, 
\end{equation}
where $c_\tau$ is a constant depending only on the conformal modulus of the torus. Using these, the connection $1$-form, which they denote by $\alpha$, 
corresponds to our connection $1$-form via 
%This choice ensures that the connection 1-form %$\alpha$ matches the one in~\cite{CJL}
$$
d\psi + c_{\tau}^{2}\frac{1}{2}(\xi_{2}d\xi_{1}-\xi_{1}d\xi_{2}) = d\theta + y_{1}dy_{2}.
$$
Note that our previous equation \eqref{GH} has a scaling constant $\beta$, which has been absorbed by rescaling $y_1, y_2$ for simplicity. We continue
to use our same coordinates $r,\theta, y_1, y_2$ below.

We now review the recipe in \cite{CJL} for constructing the semiflat metric. To simplify the computations and presentation, specialize to the case of a square torus, i.e., 
set $\tau = 1$; with this choice, the constants $a_\tau$ and $b_\tau$ appearing in that paper take the values $1$ and $0$, respectively. (The more general case 
$\tau\neq i$ proceeds similarly.)

For this torus, the semi-flat metrics in~\cite{CJL} are written using the following four variables: the first two,
\[
X_{1}=\theta, \ X_{2} = r y_{1},
\]
serve as the linear coordinates on the elliptic curve fibers, while 
\[
Y_{1}= r/c_{\tau}, \ Y_{2}=y_{2}/c_{\tau}.
\]
are the radial and angular polar coordinates, respectively, on the base, which is a punctured disk $\Delta^* = \Delta \setminus \{0\}$.   
We call these the semiflat coordinates. The $\bT^2$ divisor in our previous coordinates is parametrized by $X_2$ and $Y_2$, while $X_1$ 
is the angular variable in the normal bundle to the divisor and $Y_1$ is the distance to the divisor. As noted earlier, $c_\tau$ is an explicit constant 
depending on the degree of the $\bS^1$ bundle and the modulus $\tau$. 

Using these new variables, we write three families of deformations of the semiflat metric as follows:
\begin{enumerate}
\item Setting
$$
\tilde \theta= \theta+cr^{2}, \ \tilde y_{1}=y_{1}, \ \tilde y_{2}=y_{2}, \ \tilde r= r, 
$$
then the corresponding semiflat variables are
$$
\tilde X_{1}=X_{1}+cc_{\tau}^{2} Y_{1}^{2}, \ \tilde X_{2}=X_{2}, \ \tilde Y_{1}=Y_{1}, \ \tilde Y_{2}=Y_{2}, 
$$
and after some computation, the matrix variation of the hyperK\"ahler triple equals
$$
A=
\begin{bmatrix}
1 & 0&0\\
0 &1 &-c\\
0 &c&1
\end{bmatrix}, \
B=
\begin{bmatrix}
0&0 & 0\\
0&0 &c\\
0&-c&0
\end{bmatrix}
$$
This matrix $A$ is not symmetric, but can be brought into symmetric form by the hyperK\"ahler rotation
$$
U=\begin{bmatrix}
1 & 0&0\\
0 &\frac{1}{\sqrt{1+c^{2}}} &\frac{c}{\sqrt{1+c^{2}}}\\
0 &\frac{-c}{\sqrt{1+c^{2}}}&\frac{1}{\sqrt{1+c^{2}}}
\end{bmatrix} \in SO(3). 
$$
The equivalent infinitesimal variations in this new gauge, $\tilde A = UA$, $\tilde B = UB$, take the form
\begin{equation}\label{eq:sf1}
\tilde A=
\begin{bmatrix}
1 & 0&0\\
0 & \sqrt{1+c^{2}}&0\\
0 &0& \sqrt{1+c^{2}}
\end{bmatrix}, \
\tilde B=
\begin{bmatrix}
0&0 & 0\\
0&\frac{-c^{2}}{\sqrt{1+c^{2}}} &\frac{c}{\sqrt{1+c^{2}}}\\
0&\frac{-c}{\sqrt{1+c^{2}}}&\frac{-c^{2}}{\sqrt{1+c^{2}}}
\end{bmatrix}
\end{equation}
This variation is an $r$-dependent twist of one variable in the torus fibration, and hence is a (real) one-dimensional variation of the complex structure. 

\item Next, with the change of variables
$$
\tilde \theta = \theta , \tilde y_{1}=y_{1}+c r, \tilde y_{2}=y_{2}, \tilde r= r,
$$
the corresponding semiflat variables become
$$
\tilde X_{1}=X_{1}, \ \tilde X_{2}=X_{2} + cc_{\tau}Y_{1}^{2}, \tilde Y_{1}=Y_{1}, \tilde Y_{2}=Y_{2}.
$$
For this family, the hyperK\"ahler triple varies as
$$
A=
\begin{bmatrix}
1 & -c&0\\
c &1-\frac{c^{2}}{2} &0\\
0&0&1
\end{bmatrix}, \
B=
\begin{bmatrix}
0&c& 0\\
0&\frac{c^{2}}{2}&0\\
0&0&0
\end{bmatrix},
$$
which is brought into `standard' form by applying the hyperK\"ahler rotation 
$$
U=
\begin{bmatrix}
\frac{1-c^{2}/4}{1+c^{2}/4} & \frac{c}{1+c^{2}/4} & 0\\
\frac{-c}{1+c^{2}/4} & \frac{1-c^{2}/4}{1+c^{2}/4} & 0\\
0 & 0 & 1
\end{bmatrix}.
$$
to get the symmetric matrix $\tilde{A} = UA$ and other matrix $\tilde B = UB$: 
This results in the symmetric matrices 
\begin{equation}\label{eq:sf3}
\tilde A=\begin{bmatrix}
\frac{1+3c^{2}/4}{1+c^{2}/4} & \frac{-3c^{3}/4}{1+c^{2}/4} & 0\\
 \frac{-3c^{3}/4}{1+c^{2}/4} & \frac{1+c^{2}/4 + c^{4}/8}{1+c^{2}/4} & 0\\
 0 & 0& 1
\end{bmatrix}, \
\tilde B=\begin{bmatrix}
0 & \frac{c+c^{3}/4}{1+c^{2}/4}& 0\\
0& \frac{-c^{2}/2-c^{3}/8}{1+c^{2}/4} & 0\\
0& 0 & 0
\end{bmatrix}
\end{equation}
This is an $r$-dependent twist of the other variable on the torus, hence is the complementary real one-dimensional variation of complex structure. 

\item Finally, with the new variables
$$
\tilde \theta=\theta, \ \tilde y_{1}=y_{1}, \ \tilde y_{2}=y_{2}+cr, \ \tilde r=r 
$$
and corresponding semiflat variables
$$
\tilde X_{1}=X_{1}, \ \tilde X_{2}=X_{2}, \ \tilde Y_{1}=Y_{1},\  \tilde Y_{2}=Y_{2}+cY_{1}
$$
one obtains the matrix variation of the hyperK\"ahler triple
$$
A=
\begin{bmatrix}
1 & 0&-c\\
0 &1 &0\\
c&0&1-\frac{c^{2}}{2} 
\end{bmatrix}, \
B=
\begin{bmatrix}
0&0& c\\
0&0&0\\
0&0&\frac{c^{2}}{2}
\end{bmatrix}.
$$
Applying a similar hyperK\"ahler rotation as before yields
\begin{equation}~\label{eq:sf2}
\tilde A=\begin{bmatrix}
\frac{1+3c^{2}/4}{1+c^{2}/4} & 0 & \frac{-3c^{3}/4}{1+c^{2}/4}\\
0 & 1 & 0\\
  \frac{-3c^{3}/4}{1+c^{2}/4} & 0& \frac{1+c^{2}/4 + c^{4}/8}{1+c^{2}/4}
\end{bmatrix}, \
\tilde B=\begin{bmatrix}
0 & 0& \frac{c+c^{3}/4}{1+c^{2}/4}\\
0&0&  0\\
0& 0 & \frac{-c^{2}/2-c^{3}/8}{1+c^{2}/4} 
\end{bmatrix}.
\end{equation}
This  real one-dimensional twist changes the complex structure of the punctured disk on the base.
\end{enumerate}

\bigskip

\noindent {\large {\bf Correspondence between these families of deformations:}}  As stated earlier, the three-dimensional family of complex structures for the 
Calabi model is equivalent to the three-dimensional family of semiflat metrics by a family of hyperK\"ahler rotations. This is proved by a somewhat arduous and
mostly explicit computation, carried out in Appendix A of \cite{CJL}. We refer to that paper for details. The interpretation of this equivalence is that, although the 
set of apparent parameters for the space $\calP$ is six-dimensional, there is only a three-dimensional family of deformations up to gauge, i.e., up to hyperK\"ahler rotations. 

We conclude from all of this that there is a three-dimensional manifold which parametrizes the variations of the asymptotic structure of ALH* metrics. 
We represent this three-dimensional manifold by fixing a three-dimensional submanifold $\calS \subset \calP$ transverse to the action by 
hyperK\"ahler rotations on $\calP$. 

\subsection{The extended map and its linearization}
We now define a map $\calK$ which incorporates both the decaying perturbations as well as the ones coming from the variations at infinity. 
The tangent space of its domain is the sum of $L^{2}$ anti-self-dual harmonic forms and the space of matrices $\{\dot B\} = 
T_{(\mathrm{Id}, 0, 1)} \calP$, and it is constructed to have surjective differential. The (local) moduli space of hyperK\"ahler structures itself is then 
identified with the submanifold $\calK = 0$, but where $\dot B$ is restricted to lie in a $3$-dimensional slice of this parameter space 
which is transverse to the hyperK\"ahler rotation action at infinity described above. 

We now introduce notation to make the presentation of $\calK$ tidier. First, define
\[
x^{\epsilon_{1}, \epsilon_{2}, \epsilon_{3}}H_{\XX}^{k}(M):= 
x^{\epsilon_1}\Pi_{y}\Pi_{\theta} H_{\XX}^{k}(M)\oplus x^{\epsilon_2} \Pi_{y}^{\perp}\Pi_{\theta} H_{\XX}^{k}(M) \oplus  x^{\epsilon_3}\Pi_{\theta}^{\perp}H_{\XX}^{k}(M).
\]
Next, abusing (or conflating) notation slightly, identify the $6$-dimensional parameter space $\dot B_{ij}$ with the triplet of anti-self-dual forms 
$\underline{\dot B} = (\dot B_1, \dot B_2, \dot B_3)$ where $\dot B_i = \sum_{j=2}^3 \dot B_{ij} \omega_j^-$, and similarly write $\underline{\ddot A}$ 
with $\ddot A_i = \sum_{j=1}^3 \ddot{A}_{ij} \omega_j^+$.   Continuing on, fix a smooth nonnegative function $\chi(r)$ which equals $1$ 
for $r \geq C_1$ and $0$ for $r \leq C_1 - 1$, for $C_1$ sufficiently large; this will be used to localize various quantities to neighborhoods
of infinity.  Thus, for example, the triplets $\chi(r) \underline{\dot B}$ are the anti-self-dual `inputs' which parametrize the variations
of hyperK\"ahler structure at infinity.  Similarly, fix constants $\gamma_i^j$, $i = 1, 2, 3$, $j = 0, 1, 2, 3$, and define the triplet
\[
\ulw(\gamma) = (w_1, w_2, w_3),\ \ \mbox{where}\ \ w_i = \chi(r) ( \gamma_i^0 + \sum \gamma_i^j \omega_j^+) \in \Omega^0(M) \oplus \Omega^2_+(M).
\]
Finally, write 
\begin{equation}
\operatorname{Dom}_{\epsilon, k}=\{\ulv = \ulu + \ulw(\gamma):   \  \ulu \in  (x^{\epsilon}H^{k}_{\XX}\Omega^0(M)\oplus x^{\epsilon}H^{k}_{\XX}
\Omega^2_+(M))\otimes \RR^3\}.
\end{equation}
%where $\ulw$ is as above with $(\gamma_i^j) \in \RR^{12}$.  

We are now in a position to define the map
\begin{equation}
\begin{split}
\calK:  \left( \left(\operatorname{Dom}_{\epsilon, k+2}  \oplus \calH^{2}_+(X,g) \right. \right. & \left. \left. \oplus \ \calH^{2}_-(X,g) \oplus \RR^4 \right) 
\otimes \RR^3  \right)  \\[0.5ex] 
&  \longrightarrow x^{\epsilon+4, \epsilon+2, \epsilon}H^k_{\bfa}\Omega^{2}(X) \otimes \RR^{3}.
\end{split}
\label{eq:K}
\end{equation}
for any $k > 2$. It is given by 
\[
\calK( \ulv, \ulzeta, \uleta^-, \underline{\dot B}) = DD^*\ulv + \ulzeta + \underline{\ddot A} 
- F(DD^*\ulv + \underline{\ddot A}, \, \ulzeta,\ \uleta^- + \underline{\dot B}) , 
\]
where 
\begin{multline*}
F(DD^*\ulv + \underline{\ddot A}, \, \ulzeta,\ \uleta^-+ \underline{\dot B}) \\
=-\frac12 \mathcal Q( d^{+}D^*\ulv + \underline{\ddot A} + \ulzeta, d^{+}D^*\ulv + \underline{\ddot A} + \ulzeta) 
 - \frac12 \mathcal Q( \uleta^-+\underline{\dot B} , \uleta^-+\underline{\dot B}).
\end{multline*}

%In other words, the input is the sum of a decaying anti-selfdual $2$-form and another set of terms $\ulom_{0}+\sum c_{i} (\dot A_{i} + \dot B_{i})$
%corresponding to a deformation of the structure at infinity, and we solve for a correction $DD^{*}\ulu + \ulzeta$ which decays at infinity. 
%e $\ulom_{0}+d^{+}D^{*}\ulu  + \ulzeta +  \uleta^{-} +\sum c_{i} (\dot A_{i} + \dot B_{i})$ is asymptotic to at infinity which already solves the hyperKahler triple equation by the construction of $A$ and $B$ in the previous section. Hence we are only solving the additional correction term 

Now linearize $\calK$ with respect to only the first two components to get
\begin{equation}\label{eq:dK}
d\calK|_{(0,0,0, 0)}( \dot\ulv, \dot \ulzeta) = D D^{*} \dot\ulv + \dot \ulzeta, \ \ 
\dot\ulv\in \operatorname{Dom}_{\epsilon, k+2}, \ \dot\ulzeta \in \calH^{2}_+(X,g)\otimes \RR^3. 
\end{equation}
Trivializing $\Lambda^2_+$ with respect to the parallel basis $\{\omega_{i}^+\}$ identifies  $DD^{*}$ with 12 copies of the scalar Laplacian $\Delta$. 
As discussed in Section 4, each of these summands decomposes as
$$
\Delta=\Delta_{00}+\Delta_{0\perp} + \Delta_{\perp*},
$$
and there is a corresponding decomposition % $\operatorname{Dom}_{\epsilon, 2}$ has a corresponding decomposition
\begin{align*}
\operatorname{Dom}_{\epsilon, k+2} =\left( 
\{ \ulu + \ulw\right. & \left. : \ \ulu \in x^\epsilon \Pi_{y}\Pi_{\theta}H_{\XX}^{k+2}(M)\} \right.\\
& \left. \oplus x^\epsilon \Pi_{y}^{\perp}\Pi_{\theta}H_{\XX}^{k+2}(M) \oplus  x^\epsilon \Pi_{\theta}^{\perp}H_{\XX}^{k+2}(M)\right) \otimes \RR^{12}.
\end{align*}

The rescaled $b$-operator $\Delta_{00}$ has indicial roots $-1$ and $0$, so $\Delta$ is not surjective when acting on functions
which decay like $x^\epsilon$. Surjectivity is recovered by adding on the terms $\ulw(\gamma)$ which are constant near infinity. 
On the other hand, both $\Delta_{0\perp}$ and $\Delta_{\perp *}$ are fully elliptic, as we showed earlier, and hence 
are isomorphisms on this domain. 

This establishes the %surjectivity of the linearization of $\calK$. 
\begin{proposition}
The linearization
\[
d\calK: \operatorname{Dom}_{\epsilon,k+2} \oplus (\calH^{2}_+(X,g)\otimes \RR^3)  \longrightarrow \left(x^{\epsilon+4, \epsilon+2, \epsilon}H^k_{\bfa}\Omega^{2}(X))\right) \otimes \RR^{3}
\]
is surjective, with kernel isomorphic to $\RR^{12}$.
\end{proposition}
%\begin{proof}
%From Theorem~\ref{thm:main} we know that $d\calK$ is Fredholm. The kernel is identified with the part with indicial root 0, which are the 12 dimensional constant functions as the kernel of the scalar Laplacian. Note that from the definition of $u_{i}$ in~\eqref{eq:ui} these harmonic functions should be identified with 0 in the actual deformation.
%\end{proof}

\subsection{Perturbation of Tian--Yau metrics}
\begin{theorem}
If $(M,g)$ is an $ALH^{*}_{b}$ space, then the moduli space of nearby hyperKahler structures modulo hyperK\"ahler rotation has dimension
$3(10-b)$. The deformation parameters lie in an open neighborhood $\calU \times \calV$ where 
$0 \in \calU \subset \calH^{2}_-(M,g) \otimes \RR^3$ and $0 \in \calV \subset \calS \subset \{\dot{B}\}$. 
\end{theorem}

\begin{proof}
Expanding $\calK$ in a Taylor series, we can write $\calK=d\calK+ Q$, with $Q(\ulv, \ulzeta)$ the quadratically vanishing 
remainder.  We showed above that $d\calK$ is surjective, and has a $12$-dimensional nullspace consisting 
of  bounded harmonic functions, i.e., constants.   We can now apply the implicit function theorem to conclude that
the space of solutions of $\calK = 0$ is parametrized locally  by a product of sets $\calU \times \calV$ as in the
statement of this theorem.

For any given degree $1\leq b\leq 9$ of $\bS^1$ bundle over $\bT^2$, the space $H^{2}_{-}(ALH^{*}_{b})$ is $(9-b)$-dimensional, 
hence the space of decaying perturbations has dimension $\dim\{\uleta^{-}\in \calH^{2}_-(M,g)\otimes \RR^3\}=3(9-b)$.  
The full deformation space has three extra dimensions so that the overall dimension is $3(10-b)$.
\end{proof}
\begin{remark}
The Torelli theorem for $ALH^{*}$ spaces was established in~\cite{CJL3} and~\cite{hsvz2}. The surjectivity on the period domain proved in \cite{LeeLin} suggests 
that the dimension is $3(10-b)$ as well.  The proof above confirms this dimension via the hyperK\"ahler triple formalism. 
\end{remark}

\noindent\textbf{Acknowledgement:} The second author would like to thank Yifan Chen, Daniel Grieser, Yu-Shen Lin, Mark Stern and Song Sun for helpful discussions. The second author is supported by NSF DMS-2305363. 
This material is based upon work supported by the National Science
Foundation under Grant No. DMS-1928930, while the authors were in
residence at the Simons Laufer Mathematical Sciences Institute
(formerly MSRI) in Berkeley, California, during the Fall 2024
semester.

\end{document}